\documentclass[12pt]{article}
\usepackage[T1]{fontenc}
\usepackage{amsthm,amsmath,amssymb,mathrsfs,extarrows,mathtools}
\usepackage{graphicx}
\usepackage{subfig}
\usepackage{float}
\usepackage{indentfirst}
\usepackage{paralist}
\usepackage{booktabs}
\usepackage{multirow} % Required for multirows
\usepackage{multicol} % Required for multirows
\usepackage{enumerate}
\usepackage{enumitem}
\usepackage{mdwlist}
\usepackage{algorithm,algorithmicx}
\usepackage{bm}
\usepackage{comment}
\usepackage[numbers,sort&compress]{natbib}
\graphicspath{{pic/}}
\usepackage[colorlinks,linkcolor=blue,anchorcolor=blue,citecolor=red]{hyperref}
\makeatother
\newtheorem{theorem}{Theorem}[section]

\newtheorem{lemma}[theorem]{Lemma}

\newtheorem{remark}[theorem]{Remark}

\numberwithin{equation}{section}
\newcommand{\me}{\mathrm{e}}
\newcommand{\mi}{\mathrm{i}}
\newcommand{\ms}{\mathbb{S}}
\newcommand{\mr}{\mathbb{R}^2}
\newcommand{\uinc}{{u^{\rm inc}} }
\newcommand{\usc}{u^{\rm sc}}
\newcommand{\mn}{{K}}
\newcommand{\sn}{\mathcal{S}_{K}}
\newcommand{\tsn}{\tilde{\mathcal{S}}_{K}}

\begin{document}
    \title{A covariance-based reduced-order framework for solving acoustic scattering problems}
	\date{}
    \author{Shiwei Sun\thanks{\footnotesize Department of Mathematics, HKUST, Clear Water Bay, Kowloon, Hong Kong (swsunhk@ust.hk).},\; Hai Zhang\thanks{\footnotesize Department of Mathematics, HKUST, Clear Water Bay, Kowloon, Hong Kong (haizhang@ust.hk).}\; and Jinrui Zhang\thanks{\footnotesize Department of Mathematics, HKUST, Clear Water Bay, Kowloon, Hong Kong (jinruizhang@ust.hk).}
}
		
	\maketitle
	\begin{abstract}
This paper presents a physics-aware reduced-order method (ROM) for the efficient computation of wave-scattering problems. Standard model order reduction techniques, typically treating scattering as generic parameterized systems, frequently overlook the underlying physical structure, limiting their effectiveness in practice. To address this limitation, we propose an algorithmic framework that utilizes the intrinsic low-rank structure of the induced contrast source density. By modeling the incident wave as a random variable governed by a specified prior probability measure, we formulate the contrast source as a spatial random field whose covariance function captures essential spatial correlation and physical interactions. The reduced-order basis is then constructed via the Karhunen-Loève (KL) expansion, effectively extracting the dominant features from the scattering process to resolve multiple scattering scenarios. A central algorithmic contribution is the efficient reconstruction of the covariance matrix for arbitrary scatterer geometries and specified incident wave priors. To circumvent the prohibitive computational cost of assembling high-fidelity covariance matrices, we introduce a non-intrusive, physics-informed graph neural network (GNN) surrogate to learn the complex mapping from scatterer geometry to the source correlation kernel, enabling a highly efficient offline-online computational paradigm suitable for large-scale scattering configurations. Extensive numerical experiments demonstrate that the proposed framework achieves robust computational acceleration over full-order models without sacrificing accuracy. The advantages of our approach are twofold: it provides a robust acceleration framework for numerical scattering solvers that incorporate both the geometric information of scatterers and the prior information of incident waves, and it offers physical insights into the underlying multiple scattering mechanisms. 
		\vspace{1em}
		
	\end{abstract}
	
	% \maketitle
	
	\section{Introduction}
Wave scattering simulation aims to reconstruct the scattered field generated by a given scatterer under incident wave illumination. It plays a fundamental role in a wide range of scientific and engineering disciplines, including material design, radar detection, and geological exploration. This paper focuses on the two-dimensional acoustic scattering problem for sound-soft obstacles. Let \(\{D_i\}_{i=1}^{N_D}\) denote a cluster of disjoint sound-soft obstacles, with their union defined as $D:= \cup^{N_D}_{i=1} D_i$ and the total boundary as $\Gamma:= \cup^{N_D}_{i=1} \partial D_i$. For an incident wave $\uinc$, the total field $u$, composed of $\uinc$ and the scattered field $\usc$, satisfies the Helmholtz equation
\begin{equation}\label{equ:helm}
\Delta u+k^2u = 0 \quad  {\rm in}~\mr\setminus\overline{D},
\end{equation}
and the boundary condition
\begin{equation}\label{equ:boundary}
u=0 \quad {\rm on}~\Gamma.
\end{equation}
The objective is to efficiently compute the scattered field $\usc$ corresponding to the given scatterer $D$, and the incident wave $\uinc$. Numerically solving such problems relies on various discretization schemes, such as the finite-difference (FD) method \cite{FDTD-2005}, the finite element method (FEM) \cite{Berenger-1994}, the boundary element method (BEM) \cite{BEM-2006}, and emerging deep learning-based approaches \cite{maLearningBasedFastElectromagnetic2021,yaoEnhancedDeepLearning2024}. Among them, the BEM is particularly advantageous by expressing the solution in the form of a boundary integral, as it inherently satisfies the Sommerfeld radiation condition obeyed by the scattered field and reduces the problem dimension by concentrating unknowns on the boundary. However, this dimension reduction comes at the expense of generating dense system matrices, which require particular processing schemes to accelerate computation. One of the most classical and effective techniques is the fast multipole method (FMM) \cite{FMM-1992, FMM-2000}, which accelerates matrix-vector products by leveraging a hierarchical strategy to separate near and far-field interactions. Nevertheless, many practical scenarios necessitate a vast number of simulations for the same obstacle under different incident waves, leading to prohibitive computational costs. For example, the iteration methods for inverse problems often solve the full-order model repeatedly. This underscores the critical need for efficient strategies to achieve rapid many-query simulations.

Reduced-order method (ROM) is a powerful technique for simplifying complex high-dimensional mathematical systems. Its objective is to construct a lower-dimensional surrogate model that can capture the essential features of the original problem. In recent years, ROMs have been extensively studied for parametric differential equations \cite{quarteroniReducedBasisMethods2016,hesthavenCertifiedReducedBasis2016,bennerSurveyProjectionBasedModel2015} and related inverse problems \cite{gaoOptimalDesignBroadband2025}. Within the regime of wave equations, theoretical foundations such as parametric holomorphy have established dimension-independent bounds for the Kolmogorov width \cite{rivaShapeAnalyticitySingular2022,henriquezShapeUncertaintyQuantification2021}, which in turn implies dimension-independent convergence rates for model reduction. Furthermore, the studies \cite{georgakisNovelNumericalBasis2022,ganReducedOrderModelEquivalence2019} numerically confirm the low-rank nature of scattering operators. The reduced basis method (RBM), a classical strategy of ROM, aims to split the computation into an expensive, one-time offline phase and a cheap, fast online phase. Crucially, the efficiency of the online phase is realized if the assembly of the reduced system is independent of the full-order dimension and depends solely on the size of the reduced basis. This is because the reduced matrices and their associated operators are pre-computed and stored in the offline process. Pioneering work \cite{faresReducedBasisMethod2011} employed a greedy algorithm for basis selection with rigorous a posteriori error estimation. For broadband simulations, the proper orthogonal decomposition (POD) based on the singular value decomposition (SVD) of high-fidelity solution snapshots at sampled parameters is a popular alternative \cite{jiangReducedbasisBoundaryElement2019}. To construct an efficient affine expansion of the parameterized system, one line of work utilizes Taylor expansion to achieve a decoupled form \cite{zhongReducedorderBoundaryElement2024}, while another employs the empirical interpolation method (EIM) \cite{barraultEmpiricalInterpolationMethod2004} for non-intrusive, solver-agnostic ROM assembly \cite{edelLocallyAdaptiveNonintrusive2024}. Recent advances also explore non-intrusive approaches using convolutional autoencoders and interpolation techniques \cite{heModelOrderReduction2023,liNonIntrusiveReducedOrderModeling2021}. In multiple scattering, the RBM has been applied to 3D obstacles \cite{ganeshReducedBasisMethod2012} and elastic fractures \cite{henriquezReducedBasisMethod2024}, while T-matrix method as a classical ROM is adopted to solve large-scale scattering problems \cite{laiFastInverseElastic2022}. Although its convergence theory \cite{ganeshConvergenceAnalysisParameter2012} is well-established for simple shapes like spheres, its reliance on spherical harmonics as a universal basis leads to significantly deteriorated convergence for irregular geometries, especially for the near-field computation close to the boundary.

It is evident that ROMs have garnered increasing interest in scattering problems. However, many applications merely transplant generic parameter reduction techniques without deeply integrating the underlying physics of scattering. Departing from this trend, our work proposes a physics-aware ROM for scattering problems of plane wave incidence. The main work of this paper is summarized as follows.
\begin{itemize}
\item \textbf{A covariance-based reduced-order framework--CB-ROM--for solving acoustic scattering problems.} We model the incident wave as a random variable governed by an \emph{a priori} distribution. Consequently, the induced density on the boundary, hereafter referred to simply as the contrast source, is treated as a spatial random field. The statistical covariance function of the contrast source can be estimated from an ensemble of snapshots. The contrast source then admits a Karhunen-Lo\'eve (KL) expansion, where the basis functions are determined directly by the eigendecomposition of the covariance operator. This naturally yields a physics-aware ROM by adopting the dominant eigenfunctions as the reduced basis. Within this statistical framework, we rigorously establish the convergence and numerical stability of the proposed CB-ROM. Notably, the covariance function offers physical insights into the underlying multiple scattering mechanisms. Furthermore, this framework can be readily extended to other classes of wave propagation, including elastic and electromagnetic scattering.

\item \textbf{Three numerical schemes for covariance construction:} The core ingredient of the CB-ROM is the accurate and efficient construction of the covariance function for a specified scatterer and \emph{a priori} distribution of the incident waves. To achieve this, we investigate three numerical strategies. First, we employ a data-driven Proper Orthogonal Decomposition (POD) method, which constructs the empirical covariance matrix from high-fidelity contrast source snapshots generated by sampling the incident waves. While highly accurate, this approach incurs a prohibitive offline computational burden due to the repeated evaluation of the full-order model (FOM). To mitigate this exorbitant offline cost, we propose a second approach utilizing a physics-informed analytical ansatz for the covariance function, which efficiently captures the low-frequency asymptotic behavior of the contrast source. Finally, to achieve high fidelity without sacrificing efficiency, we introduce a physics-informed graph neural network (GNN) surrogate to learn the covariance matrix. This learning-based framework offers two critical advantages: (i) it provides a robust, non-intrusive mechanism to predict the covariance matrix for scatterers directly from their geometric parameterizations, and (ii) it facilitates a highly efficient offline-online computational paradigm for varying geometries, significantly broadening the practical applicability of the CB-ROM.

\item \textbf{Extensive numerical validation of the CB-ROM.} We first construct the covariance matrix using a data-driven POD approach, considering both full- and limited-aperture incidence settings. The resulting CB-ROM is then employed to compute the near fields for configurations involving single and multiple scatterers. To assess numerical stability, we evaluate the performance of the CB-ROM when subjected to a noisy covariance matrix. The numerical results demonstrate that the CB-ROM accurately solves the scattering problems while achieving a significant computational speedup over the conventional BIE solver. Subsequently, we validate the proposed graph neural network (GNN) approach. The trained GNN is utilized to predict the covariance matrices for both representative test samples and out-of-distribution (OOD) geometries. These results confirm that the GNN yields high-fidelity covariance reconstructions for unseen scatterers and exhibits strong generalization capabilities.

\end{itemize}

The rest of the paper is structured as follows. In Section \ref{Sec:Math-Model}, we give the mathematical formulation of the acoustic scattering problem and the conventional BIE method, thereby motivating the investigation of the ROMs. A statistical formulation of scattering problems is described in Section \ref{sec:covfun}, where we define the covariance function for the contrast source, and then derive its explicit form for a special case of a unit disk scatterer. In Section \ref{sec:CB-ROM}, we develop the CB-ROM based on the KL expansion of the contrast source. The convergence analysis of the CB-ROM is also included in this section. Then, in Section \ref{sec:podgnn}, three approaches, including the classical random POD method, a physics-informed formula and a novel learning-based GNN method, are introduced to construct the covariance matrix. Moreover, we investigate the numerical stability of our CB-ROM under a noisy covariance matrix. Extensive numerical examples are enclosed in Section \ref{sec:numerical}, to validate the effectiveness and efficiency of the proposed methods. Finally, some concluding remarks and perspectives for future work are given in Section \ref{sec:conclusion}.

\section{Mathematical formulation}\label{Sec:Math-Model}
\setcounter{equation}{0}
In this section, we establish the mathematical formulation of the governing acoustic scattering problem. While our exposition focuses on two-dimensional scattering by sound-soft obstacles under plane wave incidence, the proposed framework readily extends to other configurations, including sound-hard boundaries and point-source excitations. Subsequently, we introduce the scattering problem, review the classical boundary integral equation (BIE) formulation, and the related ROMs. 

\subsection{Acoustic scattering problems}\label{Sec:Direct-Problem}
Let \(D \subset \mathbb{R}^2\) be a bounded domain representing a sound-soft scatterer, and let \(\Gamma:= \partial D\) denote its boundary, equipped with the unit outward normal vector \(\nu\). We assume that the exterior domain, \(\mathbb{R}^2 \setminus \overline{D}\), is filled with a homogeneous and isotropic medium. Let the scatterer be illuminated by a time-harmonic plane wave $\uinc(x;d) =   \me^ {\mi kx\cdot d}$, where $k$ and $d\in\ms$ represent the wave number and the incident direction of the plane wave, respectively. The resulting total field $u$ satisfies the Helmholtz equation \eqref{equ:helm} and the boundary condition \eqref{equ:boundary}. It is well-known that the fundamental solution of the Helmholtz equation in the two-dimensional case is given by
\begin{eqnarray}\label{equ:green}
	G(x,y) = \frac{\mi}{4}H_0^{(1)}(k|x-y|).
\end{eqnarray}

Note that $\uinc$ satisfies the homogeneous Helmholtz equation \eqref{equ:helm}. Hence, we have
\begin{equation}\label{equ:sca}
	\begin{cases}
		\Delta\usc+k^2\usc=0\quad                                                & {\rm in}~
		\mr\setminus\overline{D},                                                                   \\
		% \frac{\partial\usc}{\partial\nu}=-\frac{\partial\uinc}{\partial\nu}\quad & {\rm on}~\Gamma.
        \usc = -\uinc\quad & {\rm on}~\Gamma.
	\end{cases}
\end{equation}
In addition, to ensure the uniqueness, the scattered field $\usc$ is typically required to satisfy the Sommerfeld radiation condition
\[
	\lim_{r\to\infty}\sqrt{r}(\partial_r\usc-\mi k\usc)=0,\quad r=|x|.
\]
% We call $u^{\infty}$ the far field pattern if 
% \begin{equation}\label{equ:farfield}
% 	\usc(x) = \frac{ \me^{\mi k |x|}}{\sqrt{|x|}} u^{\infty}(\hat{x})+\mathcal{O}\left(|x|^{-\frac{3}{2}}\right), \quad x\to \infty
% \end{equation}
% with $\hat{x}:=x/|x|\in \ms$. 

The forward problem of our interest, which arises in wave simulations and wide industrial applications, can be stated as follows.

\noindent \textbf {Many-query simulations: } Given the obstacle set $D$ and a large number of varying incident plane waves $\uinc$, determine the scattered field $\usc$ in the region of interest.
% or the corresponding far field pattern $u^{\infty}$. 

\subsection{BIE and related RBMs}\label{subsec:BIE-ROM}
Various well-established methods have been developed to solve the above forward scattering problem; see, for instance, \cite{Colton-Kress-2013, Berenger-1994, Mishchenko-1996-55} and the references therein. We pay particular attention to the BIE, in which the scattered field is expressed as a boundary integral. More specifically, the scattered field $\usc$ can be represented via the following single-layer potential
\begin{eqnarray}\label{equ:integratialrep}
	\usc(x) =\int_{\Gamma} G(x,y)\phi(y) ds_y, \quad x\in 	\mr\setminus\overline{D}
\end{eqnarray}
with an unknown density function $\phi$, or the combined layer potential
\begin{eqnarray}\label{equ:integratialrep-com}
	\usc(x) =\int_{\Gamma} \left(\frac{\partial G(x,y)}{\partial \nu(y)}+\mi\eta G(x,y)\right)\phi(y) ds_y, \quad x\in 	\mr\setminus\overline{D}
\end{eqnarray}
with a prescribed real coupling parameter $\eta$ and an unknown density function $\phi$. Unless otherwise specified, we use the single-layer representation \eqref{equ:integratialrep}. 
% Based on the asymptotic behavior of the fundamental solution, we have
% \begin{equation}\label{equ:ffp}
% 	u^{\infty}(\hat{x}) = \frac{  \me^{\mi\pi/4}}{\sqrt{8k\pi }}\int_{\Gamma}    \me^{-\mi k \hat{x}\cdot y} \phi(y) ds_y,\quad \hat{x}\in \ms. 
%  \end{equation}

Hence, the key point of solving the forward problem is to find the density function $\phi$. From the layer potential theory \cite{Colton-Kress-2013}, we are allowed to access the density function by solving
\begin{equation}\label{eq:boundary}
\int_{\Gamma} G(x,y)\phi(y;d) ds_y = -\uinc(x;d),\quad x\in \Gamma.
\end{equation}

Numerically, after adopting Nystr\"om method \cite{canino-1998} with the involved singular integral properly tackled \cite{baoSingularitySwappingMethod2024},    \eqref{eq:boundary} can be discretized as the following finite-dimensional linear system 
\begin{align}\label{eq:highf}
    \mathbb {G} \bm{\phi}=\bm{b}, \qquad \mathbb{G}\in \mathbb{C}^{N\times N}, \,\bm{b}\in \mathbb{C}^{N},
\end{align}
where the vector $\bm{\phi}\in \mathbb{C}^{N}$ to be solved consists of the values of $\phi$ at $N$ discrete points on $\Gamma$. 
% Solving the system generally incurs a computational complexity of order $\mathcal O(N^3)$. 
When multiple incident directions are involved, the cost becomes repetitive and poses a significant computational burden. The RBMs offer an effective direction to mitigate this issue. The core component of the RBMs is to assume that $\bm{\phi}$ lies in a low-rank space with a basis $\{v_l\}^K_{l=1}$ with $v_l\in \mathbb{C}^{N}$ and $K<N$. Then, there exists a coefficient vector $c_\phi\in\mathbb{C}^K$ such that
\begin{equation*}\label{eq:appro}
    \bm{\phi} = \mathbb{V} c_\phi,
\end{equation*}
where the transformation matrix $\mathbb{V}:=[v_1,\,\ldots,v_K]\in \mathbb{C}^{N\times K}$ is independent of the incident directions. 
% In this case, the high-fidelity system \eqref{eq:highf} is reduced into a lighter surrogate model 
% $\mathbb{G}\mathbb{V}c_\phi = \bm{b}.$
% \begin{equation}\label{eq:surrogatemodel}
%     % \mathbb{G}^{\rm ROM} c_\phi = \bm{b},
% \end{equation}
% with $\mathbb{G}^{\rm ROM}:= \mathbb{G}\mathbb{V}\in \mathbb{C}^{N\times L}$. 
% Once $\mathbb{V}$ is obtained, we are allowed to solve the reduced model with a computational complexity of $\mathcal O(NL^2)$, which significantly improves the efficiency if $L<<N$. 

In projection-based ROMs \cite{bennerSurveyProjectionBasedModel2015}, the low-dimensional unknown $c_\phi$ is typically obtained by enforcing a suitable orthogonality condition to ensure stability and accuracy. Specifically, consider a test matrix $\mathbb{W}\in \mathbb{C}^{N\times K}$ with full column rank, a well-known criterion for determining $c_\phi$ is to enforce the Petrov-Galerkin condition, namely,
\begin{align}\label{cbd}
    \mathbb{W}^\top(\bm{b}-\mathbb{G}\mathbb{V} c_\phi)=0,
\end{align}
which can be interpreted as a weighted residual formulation. In particular, if $\mathbb{V}$ consists of POD bases and $\mathbb{W}=\mathbb{V}$, we are led to the classical POD-Galerkin ROM
\begin{align}\label{rom1}
    \mathbb{G}^{\rm ROM} c_\phi = \mathbb{V}^\top\bm{b},\quad \mathbb{G}^{\rm ROM}=
    \mathbb{V}^\top \mathbb{G}\mathbb{V}\in\mathbb{C}^{K\times K}.
\end{align}
 Compared with \eqref{eq:highf}, it can be seen that ROM saves considerable computational cost for many-query simulations if the high-quality basis subspace $\mathbb{V}$ is obtained.
 % satisfying $K<<N$. 
% \begin{remark}\label{cost}
%     \textbf{Comparison of computation cost between adopting BIE directly and ROM:} Assume that there exists $Q$ incidence. If BIE is adopted directly, denote the number of discretized points for given  precision $\epsilon$ by $N$, since the inverse of $\mathbb{G}$ in \eqref{eq:highf} can be precomputed in the offline stage at one-time cost $\mathcal{O}(N^3)$, then for each incidence, the online computation cost is $\mathcal{O}(N^2)$.
%     For the reduced-order model \eqref{rom1}, suppose that $K$ bases are choosed, in the offline stage, firstly precompute $\mathbb{G}^{\rm ROM}$
%     at cost $\mathcal{O}(N^2K)$, and then the inverse of $\mathbb{G}^{\rm ROM}$ at cost $\mathcal{O}(K^3)$.  In the online procedure, for each incidence, we need to project the incident wave into the basis space at cost $\mathcal{O}(NK)$, then the coefficients are computed at cost $\mathcal{O}(K^2)$, totally $\mathcal{O}(NK+K^2)$. Thus it can be seen that ROM saves considerable computational cost for many large-scale scattering systems where $N>>K$.  
% \end{remark}

Before closing this section, we note that although the BIE for a smooth boundary without ROM can achieve exponential convergence as the number of discrete points increases, solving the full discretized system entails a substantial computational burden, especially in large-scale scattering problems or high-frequency problems. This change can be effectively addressed through the ROMs. Nevertheless, traditional ROMs often suffer from limited physical interpretability, lack rigorous theoretical numerical analysis, or fail to incorporate scatterer-dependent information. In the following, we propose a physics-informed ROM to address the aforementioned limitations.

% \section{Physical interpretation of ROM: correlation of density functions}
\section{Covariance function of the contrast source}\label{sec:covfun} 
\setcounter{equation}{0}
This section establishes the statistical framework for the scattering problem. We begin by defining the covariance function associated with the contrast source. To provide concrete insights into this formulation, we derive explicit expressions for the covariance function of a unit disk scatterer under both full- and limited-aperture incidence configurations.

\subsection{Statistical formulation}\label{31}
For plane wave incidence, the propagation direction is parameterized by the angle \(\alpha\), such that \(d:= (\cos \alpha, \sin \alpha)\). We model the incidence angle \(\alpha\) as a random variable uniformly distributed over a specified aperture \((a, b) \subseteq [0, 2\pi)\). Consequently, the contrast source \(\phi(y; \alpha)\) for \(y \in \Gamma\), introduced in \eqref{eq:boundary}, constitutes a spatial random field parameterized by \(\alpha\). The pointwise mean function of this field is defined as
\begin{equation}
    \mu(y) := \mathbb{E}_{\alpha \sim \mathcal{U}(a, b)}[\phi(y; \alpha)], \quad y \in \Gamma,
\end{equation}
where $\mathcal U(a,\,b)$ indicates the uniform distribution over $[a,\,b]$. The covariance function and correlation function of the contrast source are then specified as 
\begin{equation}\label{eq:covariance}
%     \mathcal C(x,y)=\mathbb{E}_{\theta\sim \mathcal U(a,\,b)}\left[\left(\phi(x;\,\theta)-\mu(x)\right)\overline{\left(\phi(y;\,\theta)-\mu(y)\right)}\right],\quad x\in\Gamma, \,y\in\Gamma
% \end{equation}
    \mathcal C(x,y)=\mathbb{E}_{\alpha\sim \mathcal U(a,\,b)}\left[\left(\phi(x;\,\alpha)-\mu(x)\right)\overline{\left(\phi(y;\,\alpha)-\mu(y)\right)}\right],\quad x\in\Gamma, \,y\in\Gamma
\end{equation}
and 
\begin{equation*}\label{eq:correlation}
    {\rm Cor}(x,y)=\frac{\mathcal C(x,y)}{\sqrt{\mathcal C(x,x)\mathcal C(y,y)}},\quad x\in\Gamma, \,y\in\Gamma.
\end{equation*}
The correlation function serves as the normalized, dimensionless counterpart to the covariance function, quantifying the spatial correlations within the random field. We emphasize that both functions depend intrinsically on the geometric properties of the scatterer and the \emph{a priori} distribution of the incident waves. 

Note that all multiple scattering effects are encoded in the covariance $\mathcal C(x,y)$. From a physical perspective, the values $\phi(x;\,\alpha)$ and $\phi(y;\,\alpha)$ are strongly correlated if $x,\,y\in \Gamma$ are close to each other, and the correlation decays as the distance of the two points increases. This essential observation reveals an inherent low-rank structure in the density functions. To extract it explicitly, we denote the spectrum for the covariance function by $(\lambda_n,\phi_n(y))$, which satisfies
\[\int_{\Gamma}\mathcal C(x,y)\phi_n(y)ds(y)=\lambda_n\phi_n(x).\]
Then, the random field $\phi(y;\,\alpha)$ admits the following Karhunen-Lo\'eve (KL) expansion,
\begin{align}\label{KL0}
    \phi(y;\,\alpha) = \mu(y)+\sum_{n=1}^{\infty}\sqrt{\lambda_n}\phi_n(y){z_n(\alpha)}, \quad y\in\Gamma,
\end{align}
where $z_n(\alpha),\,n=1,\cdots,\infty$ are independent standard normal random variables. 

Although a closed-form expression of $\mathcal C(x,y)$ for a general-shaped scatterer is unavailable, a numerical approximation for the covariance matrix, and therefore the correlation matrix, can be achieved through random sampling. To illustrate the low-rank structure of the contrast source, we present the magnitude of the correlation matrix for a pentagram-shaped scatterer and the decay of corresponding eigenvalues in Figure \ref{f2}. Here, we use the incident wave with the wave number $k=2\pi$, and discretize $\Gamma$ using 500 points. The correlation heatmap clearly captures the five distinct features of the scatterer, and the eigenvalues exhibit a pronounced exponential decay. Figure \ref{f3} visualizes the imaginary parts of eigenfunctions corresponding to the first six eigenvalues. We observe that these dominant modes encapsulate the characteristic behavior of the density function, which exhibits a significant contrast between the concave and convex regions of the obstacle. 
\begin{figure}[H]
    \centering
    \subfloat[]{
    \includegraphics[width=0.3\linewidth]{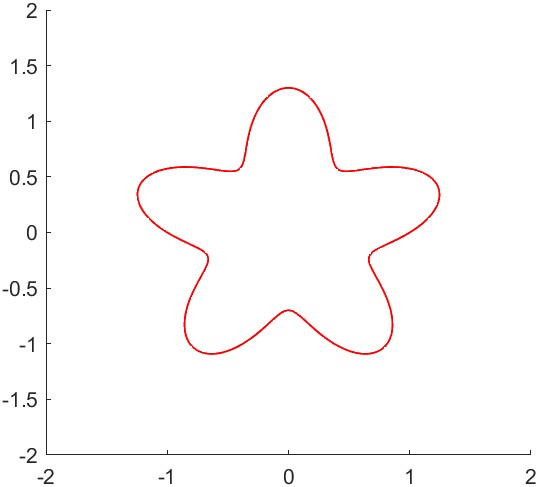}}
    \subfloat[]{
    \includegraphics[width=0.3\linewidth]{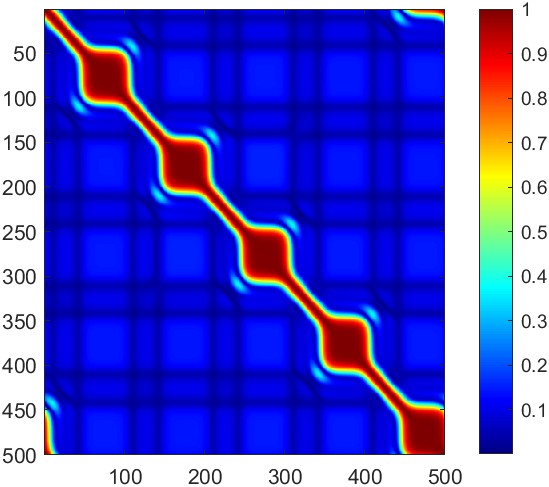}}
    \quad
    \subfloat[]{
    \includegraphics[width=0.3\linewidth]{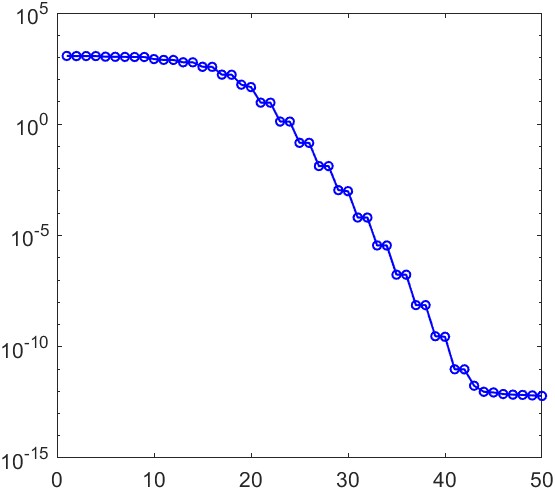}}
    \caption{(a) The pentagram-shape scatterer. (b) The magnitude of the corresponding correlation matrix. (c) The decay of the corresponding eigenvalues.}
    \label{f2}
\end{figure}

\begin{figure}[H]
    \centering
    \includegraphics[width=0.85\linewidth]{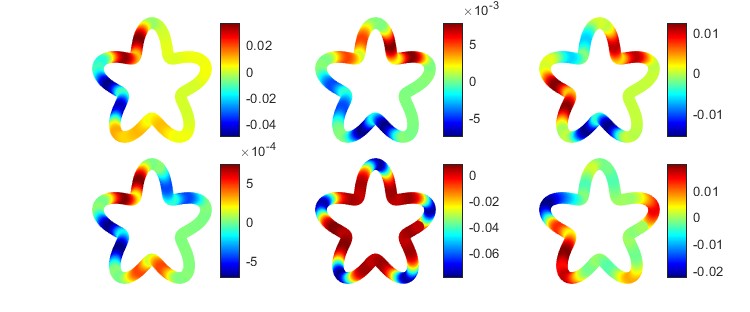}
    \caption{Principal modes: the eigenfunctions corresponding to the first six eigenvalues. }
    \label{f3}
\end{figure}

\subsection{A special case: the covariance function for a unit disk scatterer}
In this subsection, we analyze the covariance function for a special case where $D$ is a unit disk. For narrational convenience, we use $\phi(\theta; \,\alpha)$ to represent $\phi(y; \,\alpha)$ with $y=(\cos\theta,\,\sin\theta)$. In this setting, the density function admits a Fourier series expansion through the Mie scattering theory,
\[
\phi(\theta,\,\alpha) = \sum_{n=-\infty}^\infty a_n(\alpha)  \me^{\mi   n \theta},
\]
where
\[
a_n(\alpha) = A_n   \me^{-\mi n \alpha},\quad A_n=\frac{2 \mi^{n+1}}{\pi H_n^{(1)}(k)}.
\]
For later use, we define
\[
R(\theta,\,\theta') := \mathbb{E}_{\alpha\sim \mathcal U(a,b)} \left[ \phi(\theta;\alpha) \overline{\phi(\theta';\alpha)} \right]
= \mathbb{E}_{\alpha\sim \mathcal U(a,b)} \left[ \sum_{n} \sum_{m} A_n \overline{A_m} \me^{-\mi n \alpha} \me^{\mi n \theta} \me^{\mi m \alpha} \me^{-\mi m \theta'} \right],
\]
and the mean function
\[
\mu(\theta) = \mathbb{E}_{\alpha\sim \mathcal U(a,b)}[\phi(\theta;\alpha)] = \sum_n a_n  \me^{\mi   n \theta} \mathbb{E}_{\alpha\sim \mathcal U(a,b)}[ \me^{\mi  n \alpha}].
\]

Building on the preliminaries established above, we next analyze the covariance function of this scatterer under two scenarios: the covariance matrices for full-aperture and limited-aperture incidence.
\subsubsection{The full-aperture case}
We first consider the full-aperture incidence case, namely, $[a,\,b)=[0,\,2\pi)$. Then, we have
\[R(\theta, \theta') =\sum_{n} |A_n|^2   \me^{\mi   n (\theta-\theta')},\quad \mu(\theta)=A_0.
\]
Immediately, the explicit expression of the covariance function is given by
\begin{equation}\label{eq:covdisk}
  C(\theta,\,\theta') = R(\theta,\,\theta') - \mu(\theta) \overline{\mu(\theta')} = \sum_{n\neq 0} |A_n|^2   \me^{\mi   n (\theta-\theta') },\quad \theta,\,\theta' \in [0,\,2\pi).
\end{equation}
We can further derive the following KL expansion:
\begin{align}\label{KL}
    \phi(\theta;\,\alpha)=A_0+\sum_{n\neq 0}\sqrt{\lambda_n}\me^{\mi n\theta}z_n(\alpha),  \quad\lambda_n=2\pi|A_n|^2.
\end{align}

\begin{figure}[htbp]
    \centering
    \subfloat[]{
    \includegraphics[width=0.35\linewidth]{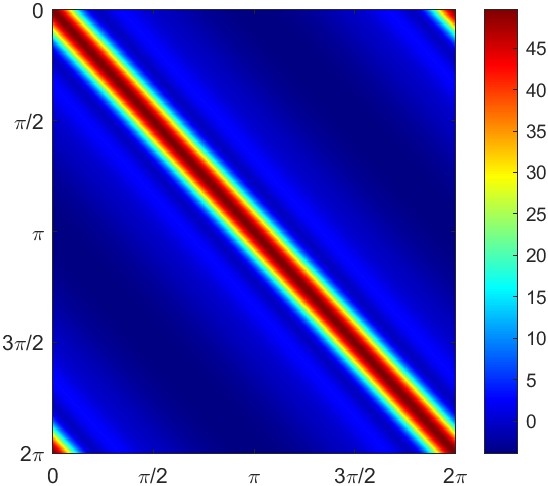}}
    \quad
    \subfloat[]{
    \includegraphics[width=0.35\linewidth]{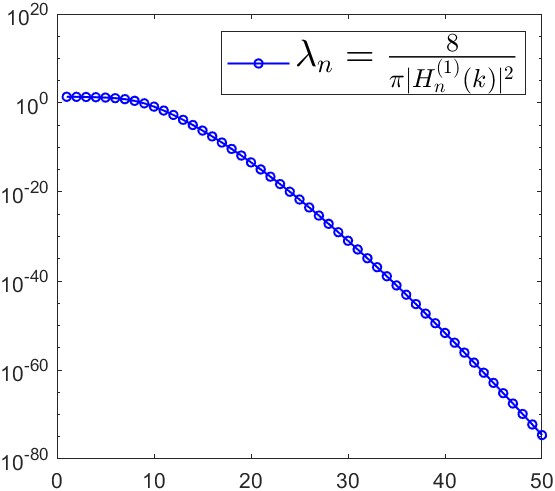}}
    \caption{(a) The magnitude of the covariance matrix \eqref{eq:covdisk} for a unit disk scatterer. (b) The eigenvalue distribution of this covariance function.}
    \label{f1}
\end{figure}

\begin{figure}[htbp]
    \centering
    \includegraphics[width=0.85\linewidth]{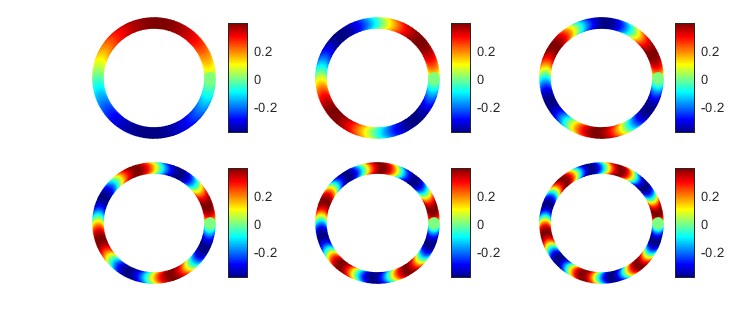}
    \caption{Principal modes: the eigenfunctions corresponding to the first six eigenvalues.}
    \label{f4}
\end{figure}
The magnitude of the covariance function \eqref{eq:covdisk} and its eigenvalues are depicted in Figure \ref{f1}. We can observe an obvious exponential decay of the eigenvalues due to the asymptotic behavior of the Hankel function $H_n^{(1)}(k)$ as $n$ increases. This decay shows a low-rank structure of the covariance function. Furthermore, the first six principal eigenfunctions are visualized in Figure \ref{f4}. It is evident that the eigenfunctions of smaller eigenvalues capture higher-frequency components of the contrast source.

\subsubsection{The limited-aperture case}
Assume that $\alpha$ obeys a uniform distribution in the intervel $(a,\,b)\Subset (0,\,2\pi)$. We first define 
\[
\Phi^{(a,b)}_{m-n}:=\mathbb{E}_{\alpha\sim \mathcal U(a,b)} [ \me^{\mi (m-n) \alpha} ] = \frac{1}{b-a} \int_a^b \me^{\mi (m-n) \alpha} d\alpha = 
\begin{cases}
\me^{\mi \frac{(a+b)(m-n)}{2} } \operatorname{sinc}\left(\frac{(m-n)(b-a)}{2}\right), & m\neq n, \\
1, & m=n.
\end{cases}
\]
Then, we have
\[
R(\theta, \theta') = \sum_{n} \sum_{m} A_n \overline{A_m} \, \Phi^{(a,b)}_{m-n} \, \me^{\mi n \theta} \me^{-\mi m \theta'},\quad \mu(\theta) =\sum_{n=-\infty}^\infty A_n \overline{\Phi^{(a,b)}_n }\me^{\mi n \theta}.
\]
As a result, the covariance function in the case of a limited aperture $(a,b)$ is
\begin{align}\label{limap}
    C(\theta, \theta') = R(\theta, \theta') - \mu(\theta) \overline{\mu(\theta')} = \sum_{n,m} A_n \overline{A_m} \left[ \Phi^{(a,b)}_{m-n} -  \overline{\Phi^{(a,b)}_n }\Phi^{(a,b)}_m \right] \me^{\mi n \theta - \mi m \theta'}.
\end{align}
Figure \ref{f5} depicts a unit disk scatterer with a limited incidence aperture $[a,\,b] =[\frac{3\pi}{4},\frac{5\pi}{4}]$ and the eigenvalue distribution of the corresponding covariance function. Compared with the full-aperture case, the rate of decay is faster, which suggests that fewer dominant modes are present in this limited-aperture scenario.
\begin{figure}[htbp]
    \centering
    \subfloat[]{
    \includegraphics[width=0.35\linewidth]{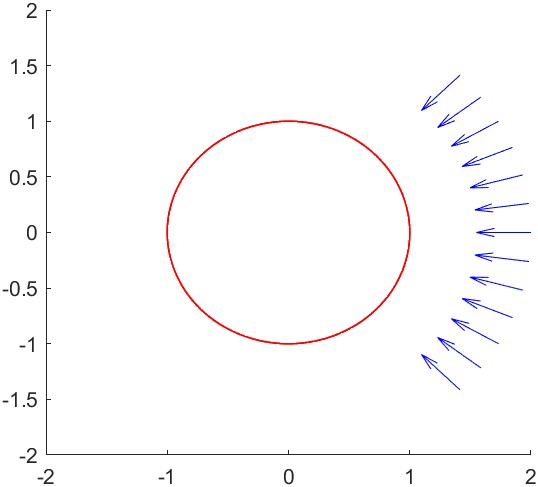}}
    \quad
    \subfloat[]{
    \includegraphics[width=0.35\linewidth]{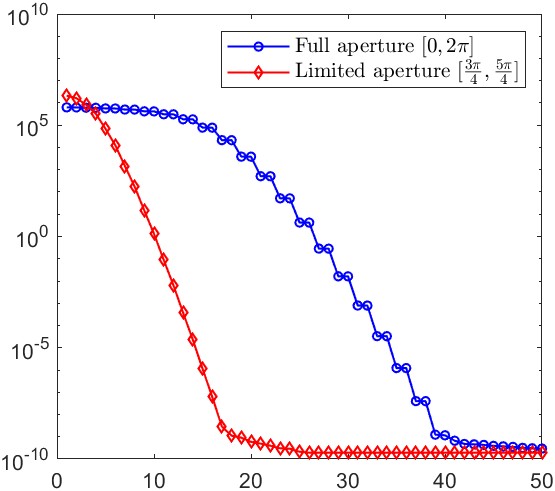}}
    \caption{(a) Unit disk with limited-aperture incidence. The blue arrows indicate the incidence range. (b) The decay of the corresponding eigenvalues compared with the full-aperture.}
    \label{f5}
\end{figure}

\begin{remark}
In the limited case that $\alpha\sim U(\alpha_0-\epsilon,\alpha_0+\epsilon)$ and let $\epsilon\rightarrow 0$. Using the second order approximation $\operatorname{sinc}(x)\approx 1-\frac{x^2}{6}$ in \eqref{limap}, we obtain that 
\[
C(\theta, \theta') \approx
% \frac{\epsilon^2}{3} \sum_{n,m} n m \, A_n \overline{A_m} \me^{\mi n (\theta - \alpha_0) - \mi m (\theta' - \alpha_0)}=
\frac{\epsilon^2}{3} \frac{\partial \phi}{\partial \alpha}(\theta;\alpha_0) \overline{\frac{\partial \phi}{\partial \alpha}(\theta';\alpha_0)}, \quad \frac{\partial \phi}{\partial \alpha}(\theta;\alpha_0) =\frac{\partial \phi(\theta;\alpha)}{\partial \alpha}\bigg|_{\alpha_0} = -\mi \sum_n n A_n \me^{\mi n (\theta - \alpha_0)},
\]
which means that only one nonzero eigenvalue remains and the corresponding eigenfunction is $\frac{\partial \phi}{\partial \alpha}(\theta;\alpha_0)$. It can be deduced that the contrast source under the incident direction near $\alpha_0$ approximately lies in the space spanned by the mean value $\phi(\theta;\alpha_0)$ and the eigenfunction $\frac{\partial \phi}{\partial \alpha}(\theta;\alpha_0)$ since  
\[\phi(\theta;\alpha_0+\delta)\approx \phi(\theta;\alpha_0)+\delta \frac{\partial \phi}{\partial \alpha}(\theta;\alpha_0).\]
\end{remark}

\section{Covariance-based reduced-order method (CB-ROM)}\label{sec:CB-ROM}
\setcounter{equation}{0}
In this section, we propose a covariance-based reduced-order method (CB-ROM), which adopts the leading eigenfunctions in the KL expansion of the contrast source as the reduced bases. Furthermore, a detailed convergence analysis of the CB-ROM is present in the general statistical framework, where the rate of convergence depends on the smoothness of the covariance function.

\subsection{Mathematical framework}
We first introduce the FOM of our scattering problem. As mentioned in Section \ref{sec:covfun}, for a specific plane wave incidence $\uinc(\alpha)$ with direction $d = (\cos \alpha,\,\sin \alpha)$, the resulting density function in \eqref{eq:boundary} can be expanded as 
\begin{equation}\label{eq:phi-expansion}
\phi(y;\,\alpha)=\sum_{n=0}^\infty \tilde a_n(\alpha)\phi_n(y),\quad y\in\Gamma
\end{equation}
where $\tilde a_0(\alpha) = 1,\,\phi_0(y) = \mu(y)$, and $\phi_n(y),\,n=1,\,2,\,\ldots$ are the eigenfunctions of the corresponding covariance function. Define $S:L^2(\Gamma)\rightarrow L^2(\Gamma)$ to be the boundary integral operator
\[[S\phi](x) = \int_\Gamma G(x,y)\phi(y)ds(y), \quad x,\,y\in\Gamma.\]
Then, these unknown coefficients $\{\tilde a_i(\alpha)\}^\infty_{i=1}$ in\eqref{eq:phi-expansion} can be determined by solving an infinite-dimensional linear system 
\begin{align}\label{fm}
    \sum_{n=0}^{\infty}\left(\phi_j,\,S\phi_n\right)\tilde a_n(\alpha) = \left(-\uinc(\alpha),\phi_j\right),\quad j=0,\,1,\,\cdots,\,\infty,
\end{align}
where $(\cdot,\,\cdot)$ denotes the inner product in $L^2(\Gamma)$. 

The essence of our CB-ROM is to approximate the contrast source with the first ${K}$ eigenfunctions. Specifically, we replace the true contrast source with the truncated version
\begin{equation}\label{eq:phi-expansion-truncated}
    \phi^{\rm ROM}(y;\,\alpha)=\sum_{n=0}^{{K}-1}a_n^{\rm ROM}(\alpha)\phi_n(y),\quad y\in\Gamma.
\end{equation}
The unkonwn coefficients $\{a_n^{\rm ROM}(\alpha)\}^{{K}-1}_{n=0}$ are then determined by the following Galerkin system
\begin{align}\label{fm2}
    \sum_{n=0}^{{K}-1}\left(\phi_j,\,S\phi_n\right)a_n^{\rm ROM}(\alpha) = (-\uinc(\alpha),\phi_j),\quad y\in \Gamma,\,j=0,\,1,\,\cdots,\,{{K}-1}.
\end{align}
The above equation can be parameterized as  
\begin{align}\label{feq}
    T_{K}{\bm a}^{\rm ROM} = \bm{b},
\end{align}
where $T_{K}\in \mathbb{C}^{{K}\times {K}}$ with $(i,n)$-th entry $\left(\phi_j,\,S\phi_n\right)$, ${\bm a}^{\rm ROM} = [a_0^{\rm ROM}(\alpha),\,\cdots,\,a_{{K}-1}^{\rm ROM}(\alpha)]^T\in \mathbb{C}^{{K}}$ and $\bm b = [b_0(\alpha),\,\cdots,\,b_{{K}-1}(\alpha)]^T\in \mathbb{C}^{{K}}$ with $b_j(\alpha)=-\left(\uinc(\alpha),\phi_j\right)$. In practice, we need to discretize the boundary $\Gamma$ properly, after which the continuous formulation is reduced to the discrete linear system \eqref{cbd}.

\subsection{Convergence analysis}
In this subsection, we investigate the convergence analysis of the CB-ROM within the statistical framework. Recall that for plane wave incidence with random direction angle $\alpha$ uniformly distributed in $(a,\,b)\subseteq(0,\,2\pi)$, the yielding contrast source admits a KL expansion
\begin{align*}
    \phi(y;\,\alpha) = \mu(y)+\sum_{n=1}^{\infty}\sqrt{\lambda_n}\phi_n(y)z_n(\alpha)=:\sum_{n=0}^{\infty}\hat a_n(\alpha)\phi_n(y), \quad y\in\Gamma,
\end{align*}
where $\hat a_0(\alpha)=1,\,\phi_0(y)=\mu(y)$, and for $n=1,\,2,\,\ldots,\,\infty$, $\hat a_n(\alpha)$ are independent random functions of $\alpha$ with zero mean and variance $\lambda_n$. Accordingly, our CB-ROM approximation using the first ${K}$ dominant eigenfunctions is given by
\begin{equation}\label{eq:phi-ex-trun-ran}
    \phi^{\rm ROM}(y;\,\alpha)=\sum_{n=0}^{{K}-1}\hat a_n^{\rm ROM}(\alpha)\phi_n(y),\quad y\in\Gamma,
\end{equation}
where $\hat a_n^{\rm ROM}(\alpha),\,n=0,\,\ldots,\,{K}-1$ are random functions of $\alpha$. Then, we have the following convergence result.
\begin{theorem}
Let $\phi(y;\,\alpha)$ be the contrast source corresponding to plane wave incidence with random direction angle $\alpha \sim \mathcal{U}(a,b)$, and let $\phi^{\mathrm{ROM}}(y;\,\alpha)$ be its approximation obtained by the CB-ROM using the first ${K}$ dominant eigenfunctions. Then, we have
    \begin{align}\label{cva}
    \mathbb{E}_{\alpha\sim(a,\,b)} \left[\|\phi^{\rm ROM}(y;\,\alpha)-\phi(y;\,\alpha)\|_{L^2(\Gamma)}^2\right]\leq ({{K}\|T_{{K}}^{-1}\|^2_2}\|S\|^2+1)\varepsilon_{{K}}, \quad
    \varepsilon_{{K}} = \sum_{n={{K}}}^{\infty} \lambda_n,
\end{align}
where $\|S\|$ is the operator norm of the single-layer potential, and $\|T_{{K}}^{-1}\|_2$ denotes the spectral norm of the matrix $T_{{K}}^{-1}$. Both $\|S\|$ and $\|T_{{K}}^{-1}\|_2^2$ depend on the geometric complexity of the scatterer.
\end{theorem}

\begin{proof}
For brevity, we omit the dependence of all functions on $y\in\Gamma$ and $\alpha\sim\mathcal U(a,\,b)$ throughout this proof. Moreover, $\mathbb{E}$ indicates taking the expectation with respect to $\alpha$. The projection of the true density function onto the subspace spanned by $\{\phi_n,\,n=0,\cdots,{K}-1\}$ is denoted as 
\[\phi_{{K}}:=P_{{K}}\phi=\sum_{n=0}^{{K}-1} \hat a_n\phi_n.\]
Using the triangle inequality, we have
     \begin{eqnarray*}
     \mathbb{E} \left[\|\phi^{\rm ROM}-\phi\|_{L^2(\Gamma)}^2\right]
     &=& \mathbb{E} \left[\|\phi-\phi_{{K}}\|_{L^2(\Gamma)}^2+\|\phi_{{K}}-\phi^{\rm ROM}\|_{L^2(\Gamma)}^2\right]\\
     &=& \mathbb{E}\left[\|\phi-\phi_{{K}}\|_{L^2(\Gamma)}^2\right] + \mathbb{E}\left[\|\phi_{{K}}-\phi^{\rm ROM}\|_{L^2(\Gamma)}^2\right]. 
    \end{eqnarray*}
We next analyze the two terms on the right separately. 

The estimate for the first part can be obtained directly from the KL expansion, namely,
\begin{eqnarray}\label{fp}
     \mathbb{E} \left[ \left\| \phi - \phi_{{K}} \right\|_{L^2(\Gamma)}^2 \right] = \mathbb{E} \left[ \int_{\Gamma}\left(\sum_{n={K}}^{\infty} \hat a_n \phi_n\right)^2  ds(y) \right]= \mathbb{E} \left[ \left(\sum_{n={K}}^{\infty} \hat a_n\right)^2 \right] = 
\sum_{n={K}}^{\infty} \lambda_n=:\varepsilon_{{K}},
\end{eqnarray}
since $\{\hat a_n\}^\infty_{n=1}$ are uncorrelated random variables with zero mean and variance $\lambda_n$.

For estimating the second part, we define
$$r_{{K}}^j=\sum_{n=0}^{{K}-1}(\phi_j,\,S\phi_n)\hat a_n-\sum_{n=0}^{\infty}(\phi_j,\,S\phi_n)\hat a_n,\quad j=0,\,1,\,\cdots,\,\infty.$$
Then, 
\[ \mathbb{E} \left[| r_{{K}}^j|^2\right]=\mathbb{E} \left[\left\|(\phi_j,S(I-P_{{K}})\phi)\right\|^2\right]\leq
\mathbb{E} \left[\|\phi_j\|^2_{L^2(\Gamma)}\|S\|^2\|(I-P_{{K}})\phi\|^2_{L^2(\Gamma)}\right]= \|S\|^2\varepsilon_{{K}}. \]
It can be seen from the full-order model that 
\begin{align}\label{apro}
    \sum_{n=0}^{{K}-1}(\phi_j,\,S\phi_n)\hat a_n = -(\uinc,\phi_j)+r^j_{{K}},\quad j=0,\,\cdots,\,{K}-1.
\end{align}
Assemble these ${{K}}$ equations in matrix form
\[T_{{K}}\hat{\bm a}_{{K}} = \bm b+\bm r_{{K}},\]
where $\hat{\bm a}_{{K}}=[a_0,\,a_1,\ldots,\,a_{{K}-1}]^T$, $\bm r_{{K}} = [r^0_{{K}},\,\ldots,\,r^{{K}-1}_{{K}}]^T$ and $\bm b = [-(\uinc,\phi_0),\,\ldots,\,-(\uinc,\phi_{{K}-1})]^T$.
Note that 
\[T_{{K}}\hat{\bm a}^{\rm ROM} = \bm b,\]
where $\hat{\bm a}^{\rm ROM} = [\hat a^{\rm ROM}_0,\,\ldots,\,\hat a^{\rm ROM}_{{K}-1}]$. We then have
\[T_{{K}}(\hat{\bm a}_{{K}}-\hat{\bm a}^{\rm ROM})=\bm r_{{K}}.\]
Consequently,
\begin{eqnarray}\label{sp}
    \mathbb{E} \left[\|\phi_{{K}}-\phi^{\rm ROM}\|_{L^2(\Gamma)}^2\right]&=&\mathbb{E} \left[\|\hat{\bm a}_{{K}}-\hat{\bm a}^{\rm ROM}\|_2^2\right] = \mathbb{E} \left[\|T_{{K}}^{-1}\bm r_{{K}}\|_2^2\right]\\ \nonumber
    &\leq& \|T_{{K}}^{-1}\|^2_2\mathbb{E} \left[\|\bm r_{{K}}\|_2^2\right]\leq {{K}}\|T_{{K}}^{-1}\|^2_2\|S\|^2\varepsilon_{{K}}.
\end{eqnarray}
% with $\|\cdot\|_2$ the spectral norm of the matrix.
We are ready to obtain \eqref{cva} by combining \eqref{fp} and \eqref{sp}. The proof is complete.
\end{proof}
This theorem reveals that the convergence of the CB-ROM in the sense of expectation is directly related to the tail sum of eigenvalues of the covariance function. To understand it intuitively, we compare the numerical errors and theoretical decaying factor $\sqrt{\varepsilon_\mn}=\sqrt{\sum_{n={K}}^N\lambda_n}$ for different shapes and apertures in Figure \ref{f8}, where $N$ denotes the number of points to discretize the boundary integral operator in numerical experiments. In each case, the near field is calculated for a specific incidence plane wave with wavenumber $k=2\pi$ and direction angle $\alpha$ is randomly sampled with $\alpha\sim\mathcal U(\pi,\frac{3\pi}{2})$. The numerical results show that the observed errors decay at almost the same rate as the tail sum of eigenvalues $\sqrt{\varepsilon_\mn}$, which is strongly consistent with the above convergence analysis.
\begin{figure}[htbp]
    \centering
    \subfloat[]{\includegraphics[width=0.25\linewidth]{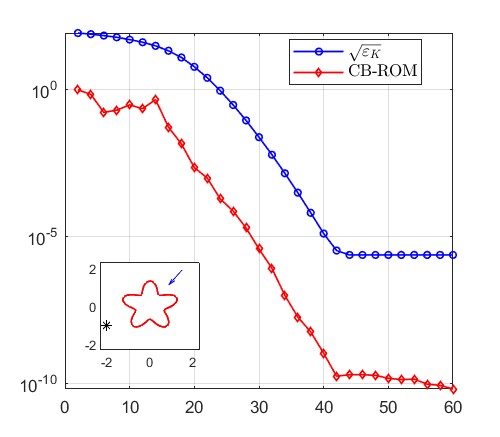}}
    \subfloat[]{\includegraphics[width=0.25\linewidth]{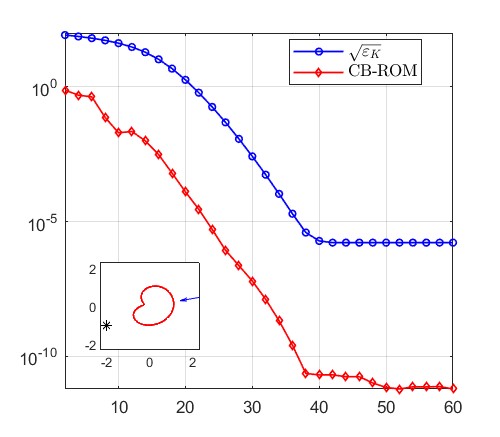}}
    \subfloat[]{\includegraphics[width=0.25\linewidth]{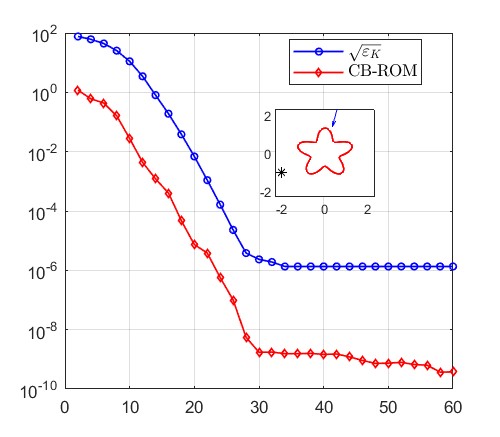}}
    \subfloat[]{\includegraphics[width=0.25\linewidth]{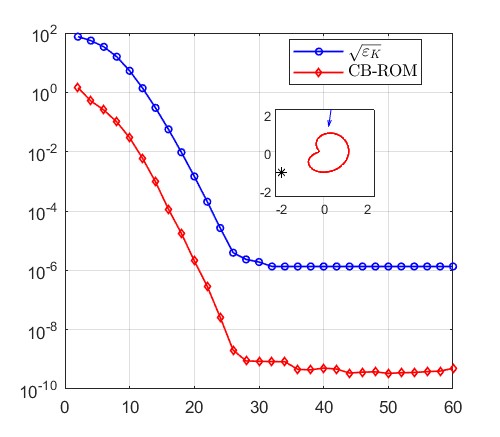}}
    \caption{The comparison of the near-field error at $(-2,-1)$, indicated by a black star, by the CB-ROM against the decaying factor {$\sqrt{\varepsilon_\mn}$}. (a) and (b) show the results for the star-shaped and apple-shaped scatterer, respectively, under the full-aperture incidence. (c) and (d) are the corresponding results under the limited aperture $(\pi,2\pi)$.}
    \label{f8}
\end{figure}

\section{Approaches to construct the covariance matrix}\label{sec:podgnn}
\setcounter{equation}{0}
The successful implementation of the CB-ROM relies on obtaining the covariance function for a specified scatterer geometry and incident wave distribution. To this end, this section is devoted to numerical approaches for constructing the covariance function associated with an arbitrary scatterer. We first introduce the data-driven POD, a widely used technique for reduced basis computation. Then we provide a simple physics-informed covariance function for multiple scatterers to capture the low-rank property through controlling the correlation length. To enable fast and accurate covariance construction for general scatterers, we further develop a learning-based GNN method that captures the correlation of the contrast source directly from geometric information. Finally, we give the numerical stability analysis of the CB-ROM in the presence of a noisy covariance function.

Before presenting the details, we highlight that the covariance function \eqref{eq:covariance} for a given scatterer $D$ is numerically approximated by the covariance matrix, denoted as $\Sigma(D)$, after discretizing the boundary $\Gamma$ with $N$ nodes $\{y_i\}^N_{i=1}$. In this section, the contrast source in \eqref{eq:boundary} is understood in its discretized form, namely, $\hat \phi(\alpha):=[\phi(y_1;\,\alpha),\,\ldots,\,\phi(y_N;\,\alpha)]^T$. 
\subsection{Data-driven random POD}\label{POD}
POD typically computes the covariance function using the method of snapshots. Specifically, we first sample multiple incident angles of the plane wave uniformly at random in the given aperture $(a,\,b)$, denoted as $\{\alpha_i,\,i=1,\,\ldots,\,M\}$. Using these plane waves as incident fields $\uinc$, we collect high-fidelity snapshots $\{\hat \phi(\alpha_i),\,i=1,\,2,\,\cdots,\,M\}$ by sovling the FOM \eqref{fm}. The sample covariance matrix is then defined by
\begin{align}\label{eq:cov-num}
    \Sigma_M(D)=\frac{1}{M-1}\sum_{i=1}^M(\hat \phi(\alpha_i)-\mu)(\hat \phi(\alpha_i)-\mu)^H\in\mathbb{C}^{N\times N},
\end{align}
where $\mu = \frac{1}{M}\sum_{i=1}^M \hat \phi(\alpha_i)$ is the sample mean vector, and $(\cdot)^H$ indicates the Hermitian transpose. By the central limit theorem, $\Sigma_M(D)$ concentrates around the true covariance matrix as $M$ grows.

Once the sample covariance matrix $\Sigma_M(D)$ is constructed, we perform the SVD to get
\[\Sigma_{M}(D)=\mathbb{U}\mathbb{S}\mathbb{V}^H,\]
where $\mathbb{U}=[\zeta_1;\,\zeta_2;\,\cdots;\,\zeta_{N}]\in\mathbb{C}^{N\times N}$ and $\mathbb{V}=[\psi_1;\,\psi_2;\,\cdots;\,\psi_M]\in\mathbb{C}^{M\times M}$ are orthogonal matrices. Numerically, the first $({K}-1)$ columns of $\mathbb{U}$, together with the mean vector $\mu$, form ${K}$ reduced bases of the CB-ROM.

\subsection{Physics-informed intuitive formula}\label{subsec:pi-formula}
In this subsection, we propose a physics-informed intuitive formula of the covariance function for a single scatterer,
\begin{equation}\label{pf}
\mathcal C(x,y) = \operatorname{exp}\left(-\frac{\|x-y\|_2^2}{\beta^2\lambda^2}\right),
\end{equation}
where $\lambda$ is the wavelength, $\beta>0$ is the scaling factor controlling the correlation length. For the case of multiple scatterers, namely, $D = \cup^{N_D}_{i=1} D_i$, we set
\begin{equation}
\mathcal C(x,y) = 
\begin{cases}
    \operatorname{exp}\left(-\frac{\|x-y\|^2}{\beta^2\lambda^2}\right),&\quad x,\,y\in \partial D_i,\,i=1,\,\cdots, N_D,\\
    \operatorname{exp}\left(-\frac{\|O_i-O_j\|^2}{\beta^2\lambda^2}\right), &\quad x\in \partial D_i, \,y\in \partial D_j,\,i\neq j,
\end{cases}
\end{equation}
where $O_i$ defines the center of the scatterer $D_i$. 

This analytical approximation is motivated by the physical observation that the correlation between two points of the contrast source on the boundary decays rapidly as their separation distance becomes large relative to the wavelength. As demonstrated by the numerical examples in Section \ref{subsec:PI-formula-num}, this formulation effectively captures the low-rank structure of the covariance function, yielding accurate results for both single and multiple scatterers with relatively smooth boundaries. Furthermore, for multiscale scatterers, it remains capable of globally representing the dominant low-frequency modes, thereby achieving a favorable balance between satisfactory accuracy and computational efficiency.

\subsection{The learning-based method--GNN}\label{sec:gnn}
Although the data-driven POD can yield an accurate covariance matrix for a given obstacle, this method requires repeatedly solving a linear system with a fixed size, making it computationally prohibitive. Moreover, the method is inherently case-dependent, necessitating a full recomputation for each new scatterer. In practice, the total time required to construct the covariance and the resulting ROM may even exceed the direct BIE, thereby undermining the practical value of the ROM. The intuitive physics-informed formula, while effective under certain conditions, lacks the fidelity and accuracy required for broader applications. These limitations motivate the development of an efficient and high-fidelity reconstruction strategy for the covariance matrix. 

We begin by highlighting two observations. First, the covariance matrix essentially captures the correlations among the contrast source evaluated at different discrete points, hereafter referred to as point sources for notational convenience. Intuitively, the covariance matrix indicates which point sources are likely to be simultaneously or probabilistically activated under illumination by one plane wave. Such correlations depend solely on the scatterer itself. This suggests that the covariance between any two point sources on the boundary can be inferred from their local geometric features, such as their locations and slopes. Accordingly, it is feasible to directly reconstruct the covariance matrix for a given scatterer from its geometric description. Second, the number of point sources used to discretize the boundaries is not fixed, but rather varies according to the specific geometry of each scatterer. In recent years, deep learning techniques have emerged as a powerful paradigm for tackling various mathematical problems with practical applications; see, for instance, \cite{Li-2024-40, Zou-2025-47, Sun-2025-4, Raissi-2019-378} and the references therein. A natural idea for our task is to employ deep learning to learn the covariance matrix associated with each scatterer. However, conventional architectures such as multilayer perceptrons (MLPs) and convolutional neural networks are primarily designed to handle Euclidean input data, specifically, fixed-size vectors or images whose pixels are aligned on a Cartesian grid. Consequently, these standard models cannot be directly applied to our problem, which inherently involves a variable number of boundary points. It is noteworthy that the set of point sources naturally forms graph-structured data, where the point sources and their correlations can be interpreted as vertices and edges of a graph. Graph neural networks (GNNs) \cite{caoPhysicsDataHybridDrivenEdgeFeatured2025, Wu-2020-32, Scarselli-2008-20}, which have demonstrated remarkable success in handling the complexity of graph data, are thus well-suited for learning the interaction of these point sources, enabling effective approximation of the covariance matrix.
\begin{figure}[htbp]
    \centering
    \includegraphics[width=0.35\linewidth]{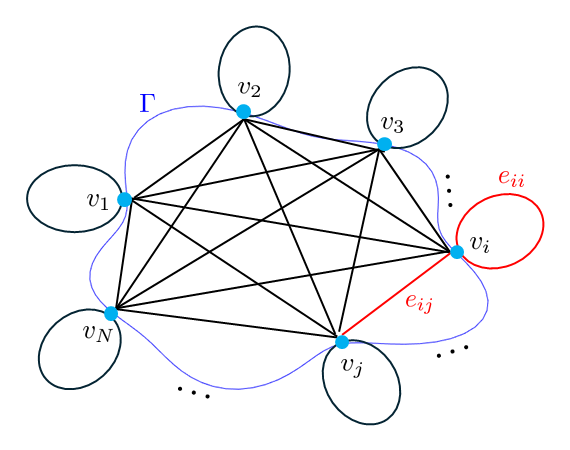}
    \caption{Discretize $\Gamma$ with $N$ nodes, and then generate a fully connected undirected graph with self-loops. }
    \label{fig:illustration-graph}
\end{figure}

We first recall basic concepts of graphs. A graph is typically represented as $\mathcal G: = (\mathcal E, \mathcal V)$, where $\mathcal V$ is the set of nodes $\{v_1,\,v_2,\,\ldots,v_{N}\}$, and $\mathcal E: =\{e_{ij}= (v_i,\,v_j)\in \mathcal V\times \mathcal V\}$ represents the edge set. We denote by ${K}(v_i):=\{v_j: e_{ij} \in \mathcal E\}$ the one-hop neighbors. Each node $v_i \in \mathcal V$ has a node feature vector $\bm x^v_i\in \mathbb R^{N_v}$, which describes the explicit contribution of $v_i$. In the graph-structured framework $\mathcal G$, each node $v_i\in \mathcal V$ is characterized by not only its explicit feature vector $\bm x^v_i$, but also a latent feature (or hidden feature) vector $\bm h^v_i\in \mathbb R^{N_h}$, which is obtained from its neighboring nodes ${K}(v_i)$. Analogously, we can assign for each edge $e_{ij}\in \mathcal E$ an edge feature vector $\bm x^e_{(i,j)}\in\mathbb R^{N_e}$. The fundamental principle of GNNs is to iteratively update these latent features by aggregating information across the graph structure, thereby capturing complex dependencies among vertices. These learned latent features can subsequently serve as the foundation for downstream tasks, such as node classification or the construction of a knowledge graph \cite{Wu-2020-32, Schlichtkrull-2018}.

Our goal is to train a GNN to construct the covariance matrix for an arbitrary given scatterer. It is reformulated as an edge-level GNN problem, in which node and edge features are updated alternately, and the final edge representations are used to construct the output covariance matrix. Assume that we discretize the boundary $\Gamma$ with $N$ nodes in BIE formulation, as depicted in Figure \ref{fig:illustration-graph}. The corresponding node set is denoted as $\mathcal E=\{v_1,\,v_2,\,\ldots,\,v_N\}$. Since the covariance matrix encodes pairwise relationships among all point sources, we construct a fully connected undirected graph with self-loops, namely, every pair of two vertices is connected by one edge. 

In GNNs, the information aggregation strategy plays an important role in the whole story, and considerable research has been devoted to designing various effective aggregation schemes; see the reviews \cite{Wu-2020-32, zhou-2020-1} for a comprehensive overview. In our numerical implementations, we first encode the initial node  and edge features into a latent space by MLPs, namely, for $i,\,j = 1,\,2,\,\cdots,N$,
\begin{equation}\label{eq:mlp-encoder}
    \bm h^v_i(0) = \sigma\left( W^v_2\cdot\sigma(W^v_1 v_i + b^v_1) + b^v_2\right), \,\bm h^e_{ij}(0) = \sigma\left( W^e_2\cdot\sigma(W^e_1 \bm x^e_{(i,j)} + b^e_1) + b^e_2\right),
\end{equation}
where $\sigma(\cdot)$ denotes a nonlinear activation function. The learnable parameters are
\[
W^v_1\in \mathbb{R}^{N_h\times N_v},\, W^e_1\in \mathbb{R}^{N_h\times N_e},\,W^v_2, W^v_2 \in \mathbb{R}^{N_h\times N_h},\, b^v_1,\, b^v_2,\, b^e_1,\, b^e_2\in \mathbb{R}^{N_h}.
\]
Subsequently, the information aggregation is carried out iteratively. At each iteration, the node features are updated by aggregating information from neighboring nodes and the corresponding edge features, namely,
\begin{equation}\label{eq:aggragationnode}
    \bm h^v_i(d) = \sum_{j\in{K}(v_i)}f\left(\bm h^v_i(d-1),\,\bm h^v_j(d-1),\, \bm h^e_{ij}(d-1)\right),\quad d=1,\,2,\,\cdots,\,D,
\end{equation}
where $f(\cdot)$ is an MLP with learnable parameters. The edge features are then updated  based on the newly computed node representations:
\begin{equation}\label{eq:aggragationedge}
    \bm h^e_{ij}(d) = g\left(\bm h^v_i(d),\,\bm h^v_j(d),\, \bm h^e_{ij}(d-1)\right),\quad d=1,\,2,\,\cdots,\,D,
\end{equation}
with $g(\cdot)$ also an MLP. After $D$ iterations, the final edge features $\bm h^e_{ij}(D)$ are decoded into a two-dimensional vector $\bm s_{ij}\in \mathbb{R}^2$ via an MLP decoder:
\begin{equation}\label{eq:mlp-decoder}
    \bm s_{ij} = W^s_2\cdot\sigma(W^s_1 \bm h^e_{ij}(D) + b^s_1) + b^s_2, \quad i,\,j = 1,\,2,\,\cdots,N,
\end{equation}
with learnable parameters $W^s_1\in \mathbb{R}^{N_h\times N_h},\,W^s_2\in\mathbb{R}^{2\times N_h},\,b^s_1\in\mathbb{R}^{N_h},\, b^s_2\in \mathbb{R}^{2}$. The two components of $\bm s_{ij}$ serve as the real-part and imaginary-part variances of two point sources at $v_i$ and $v_j$. Denote by $\mathcal F_\Theta$ the proposed GNN architecture, where $\Theta$ indicates the collection of all learnable parameters in \eqref{eq:mlp-encoder}--\eqref{eq:mlp-decoder}. For a given scatterer $D$, the output of the network is a matrix $\mathcal F_\Theta(D)\in \mathbb{R}^{N\times N}$, whose $(i,j)$-th entry is $\bm s_{ij}$. Consequently, the objective of the proposed GNN is to learn an optimal $\Theta$ such that $\mathcal F_\Theta(D)$ accurately constructs the true covariance matrix associated with the scatterer $D$.

\begin{comment}
\begin{eqnarray}\label{equ:integratialrep}
	\usc(x) =\int_{\Gamma} \left(\frac{\partial G(x,y)}{\partial \nu(y)}+\mi\eta G(x,y)\right)\phi(y) ds_y, \quad x\in 	\mr\setminus\overline{D}
\end{eqnarray}
with a prescribed real coupling parameter $\eta$ and an unknown density function $\phi$.
\end{comment}

\begin{remark}
The aggregation mechanism of GNNs is similar to the convolution operation of CNNs, which aims to integrate information from neighboring nodes. However, CNNs are designed to operate on structured, grid-like data with a fixed ordering, while GNNs can handle unstructured, order-independent data. Another notable advantage of GNNs lies in their ability to accommodate inputs of varying sizes. In the proposed GNN architecture, for example, all learnable parameters depend solely on the dimensions of node and edge feature vectors, which remain constant regardless of the number of discrete points on the boundary. As a result, the trained model can be directly applied to different discretizations, which suggests the flexibility and extensibility of our strategy.
\end{remark}

\subsection{Numerical stability analysis}\label{stable}
In practice, the exact covariance matrix is rarely available, making it essential to examine the numerical stability of the CB-ROM in the presence of the noise. In this subsection, we give a stability analysis of the CB-ROM when the numerical covariance matrix is subject to perturbation. Specifically, let $ C\in \mathbb{C}^{N\times N}$ be the true covariance matrix and $ {C}_\delta\in \mathbb{C}^{N\times N}$ be its corresponding perturbation. Define $\tilde C = C+C_\delta$ and perform the SVD to have
\begin{align}\label{pert}
C = \mathbb{U}\mathbb{S}\mathbb{V}^H,\quad \tilde{C}=\tilde{\mathbb{U}}\tilde{\mathbb{S}}\tilde{\mathbb{V}}^H,
\end{align}
where $\mathbb{U}=[\zeta_1;\,\cdots;\,\zeta_N],\,\tilde{\mathbb{U}}= [\psi_1;\,\cdots;\,\psi_N]$. In numerical implementation, the contrast source is assumed to lie in the subspace $\tilde{\mathcal{S}}_{K}=\operatorname{span}\{\psi_1,\,\dots,\,\psi_{K}\}$, rather than $\mathcal{S}_{K}=\operatorname{span}\{\zeta_1,\,\dots,\,\zeta_{K}\}$ with ${K}<N$. To facilitate the stability analysis, we need to make the following reasonable assumptions.
\begin{itemize}
    \item $C_\delta$ is a symmetric matrix. In practice, the constructed covariance matrix is always restricted to being a symmetric matrix.
    
    \item The singular values of $C$ are listed as $\lambda_1>\cdots>\lambda_{{K}}>\lambda_{{K}+1}>\cdots>\lambda_{N}\geq 0$, and define $\delta_{{K}}=\lambda_{{K}}-\lambda_{{K}+1}$. 
\end{itemize}
In the regime of random illumination, the true discrete contrast source $\hat \phi\in \mathbb{C}^N$ is denoted as 
\[\hat \phi = \sum_{n=1}^N a_n\sqrt{\lambda_n}\zeta_n,\]
where $a_n,\,n=1,\,\dots,\,N$ are random variables with zero mean and variance $1$. Moreover, we define the truncated form $\hat \phi_{K} = \sum_{n=1}^{K} a_n\sqrt{\lambda_n}\zeta_n$. In the meantime, the projection of $\hat \phi$ onto $\tilde{\mathcal{S}}_{K}$ is denoted as $\tilde \phi_{K} = \sum_{n=1}^{K} b_n \psi_n$. We aim to estimate $\|\tilde \phi_{K}-\hat \phi\|_2$ in the sense of expectation. To this end, we need to introduce the Davis-Kahan theorem \cite{Davis-Kahan}.   
\begin{lemma}\label{DK}
     Assume that $\tilde C$ is a symmetric matrix. We have
\[\|P_{\mathcal{S}_{K}}-P_{\tilde{\mathcal{S}}_{K}}\|_2\leq \frac{2\|C_\delta \|_2}{\delta_{K}},\]
where $\|\cdot\|_2$ denotes the spectral norm for matrices.
\end{lemma}
This estimate demonstrates that the projection error is controlled by the perturbation magnitude relative to the spectral gap. 
We are now at the point to state the core theorem.
\begin{theorem}
Under the above assumptions, we have
\begin{equation*}
    \mathbb{E}[\|\hat \phi - \tilde \phi_{K}\|_2] \leq \sqrt{\varepsilon_{K}} + \frac{2\|C_\delta \|_2}{\delta_{K}}\sqrt{\varepsilon_0},
\end{equation*}
where $\varepsilon_i = \sum_{n=i+1}^N \lambda_n$.
\end{theorem}
\begin{proof}
Note that 
\[\mathbb{E}[\|\hat \phi - \tilde \phi_{K}\|_2]\leq \mathbb{E}[\|\hat \phi - \hat \phi_{K}\|_2] + \mathbb{E}[\|\hat \phi_{K} - \tilde \phi_{K}\|_2].\]
We next estimate the two terms on the right side separately. On the one hand, we derive from Jensen's inequality that
\[
(\mathbb{E}[\|\hat \phi - \hat \phi_{K}\|_2])^2 \leq \mathbb{E}\left[\|\sum_{n={K}+1}^ N a_n\sqrt{\lambda_n}\zeta_n\|^2_2\right] = \sum_{n={K}+1}^N \lambda_n = \varepsilon_{K},
\] 
Hence, 
\begin{equation}\label{fpt}
    \mathbb{E}[\|\hat \phi - \hat \phi_{K}\|_2]\leq \sqrt{\varepsilon_{K}}.
\end{equation}

On the other hand, by Lemma \ref{DK}, we have
\begin{equation}\label{spt}
\begin{aligned}
    \mathbb{E}\left[\|\hat \phi_{K} - \tilde \phi_{K}\|_2\right] &= \mathbb{E}\left[\|P_{\mathcal{S}_{K}} \hat \phi - P_{\tilde{\mathcal{S}}_{K}} \hat \phi\|_2\right]\leq \frac{2\|C_\delta \|_2}{\delta_{K}}\mathbb{E}\left[\|\hat\phi\|_2\right]\leq  \frac{2\|C_\delta \|_2}{\delta_{K}}\sqrt{\varepsilon_0}.
\end{aligned}
\end{equation}
Combining \eqref{fpt} and \eqref{spt} yields 
\begin{align}
\begin{aligned}
    \mathbb{E}[\|\hat \phi - \tilde \phi_{K}\|_2]& \leq \sqrt{\varepsilon_{K}} + \frac{2\|C_\delta \|_2}{\delta_{K}}\sqrt{\varepsilon_0}.
\end{aligned}
\end{align}
The proof is then completed.
\end{proof}
This theorem indicates that the projection of the exact contrast source onto the subspace spanned by perturbed bases still retains a favorable convergence property, provided that the perturbation matrix is sufficiently small relative to the spectral gap. 

Furthermore, the numerical experiments in Section \ref{noise} show that the near-field error maintains a decay rate comparable to that of the noise-free case until the eigenvalue approaches the noise level. Under suitable assumptions on the smoothness of the covariance function, this behavior can be explained by the following theorem.
\begin{theorem}
Suppose that there exists some ${T}$ and $L>1$ such that $\frac{\lambda_i}{\lambda_{i+1}}>L$ for all $i>{T}$, i.e. the eigenvalues enter an exponential decay regime beyond index $T$, and $\|C_\delta\|_2 < \lambda_{K}$, then for $\mn>{T}$, 
\[
\mathbb{E}[\|\hat \phi_{K} - \tilde \phi_{K}\|_2]\leq\sqrt{\varepsilon_{K}} + C_1 \lambda_\mn +C_2 \mn \sqrt{\lambda_\mn},
\]
where $C_1=\frac{2\varepsilon_0}{\lambda_{T}-\lambda_{{T}+1}}, C_2=\frac{2L}{L-1}$ are $\mn$-independent constants.
\end{theorem}
\begin{proof}
\begin{align*}
\|\hat \phi_{K} - \tilde \phi_{K}\|_2 &= \|P_{\sn}\hat{\phi} - P_{\tsn}\hat{\phi}\|_2 \\
&= \|P_{\sn}\hat{\phi} - P_{\tsn}(\hat \phi_{K} + (\hat \phi - \hat \phi_{K}))\|_2 \\
&\leq \|P_{\tsn}(\hat \phi - \hat \phi_{K})\|_2 + \|(P_{\sn} - P_{\tsn})\hat \phi_{K}\|_2.
\end{align*}

Due to the projector $\|P_{\tsn}\| \leq 1$, $\mathbb{E}[\|\hat \phi - \hat \phi_{K}\|_2] \leq \sqrt{\varepsilon_{K}}$, the first part satisfies
\begin{align}\label{ss1}
 \mathbb{E}[\|P_{\tsn}(\hat \phi - \hat \phi_{K})\|_2] \leq \sqrt{\varepsilon_{K}}.   
\end{align}
and the following estimate of the second part holds,
\begin{align*}
\mathbb{E}[\|(P_{\sn} - P_{\tsn})\hat \phi_{K}\|_2]&\leq \mathbb{E}[\|(P_{\sn} - P_{\tsn})(\hat \phi_{T}+\hat \phi_{K}-\hat \phi_{T})\|_2]  \\
&\leq \mathbb{E}[\|(P_{\sn} - P_{\tsn})\hat \phi_{T}\|_2]+\mathbb{E}[\|(P_{\sn} - P_{\tsn})(\hat \phi_{K}-\hat \phi_{T})\|_2].
\end{align*}
With the aid of Davis-Kahan Lemma \ref{DK}, we have
\begin{align}\label{ss2}
\mathbb{E}[\|(P_{\sn} - P_{\tsn})\hat \phi_{T}\|_2]\leq \mathbb{E}[\|(P_{\mathcal{S}_{T}} - P_{\tilde{\mathcal{S}}_{T}})\hat \phi_{T}\|_2]\leq \frac{2\lambda_\mn}{\lambda_{T}-\lambda_{{T}+1}}\mathbb{E}[\|\hat \phi_{T}\|_2]:=C_1\lambda_\mn,
\end{align}
and 
\begin{align}\label{ss3}
\begin{aligned}
    \mathbb{E}[\|(P_{\sn} - P_{\tsn})(\hat \phi_{K}-\hat \phi_{T})\|_2]&\leq
\sum_{i={T}+1}^\mn \mathbb{E}[|a_i|\sqrt{\lambda_i} \cdot \|(P_{\sn} - P_{\tsn})u_i\|_2] \\
&\leq \sum_{i={T}+1}^\mn \sqrt{\lambda_i} \cdot \|(P_{\mathcal{S}_i} - P_{\tilde{\mathcal{S}}_i})u_i\|_2 
\leq \sum_{i={T}+1}^\mn \frac{2\sqrt{\lambda_i}\|C_\delta\|_2}{\lambda_i - \lambda_{i+1}}\\
&\leq \frac{2L}{L-1}\sum_{i={T}+1}^\mn \frac{\|C_\delta\|_2}{\sqrt{\lambda_i}} \leq  \frac{2L}{L-1}\sum_{i=1}^\mn  \frac{\lambda_\mn}{\sqrt{\lambda_i}} \leq C_2 \mn \sqrt{\lambda_\mn}.
\end{aligned}
\end{align}
Combining \eqref{ss1}--\eqref{ss3}, we are ready to have
\[
\mathbb{E}[\|\hat \phi_{K} - \tilde \phi_{K}\|_2]\leq\sqrt{\varepsilon_{K}} + C_1 \lambda_\mn +C_2 \mn \sqrt{\lambda_\mn}.
\]
The proof is completed.
\end{proof}

\section{Numerical experiments}\label{sec:numerical}
\setcounter{equation}{0}
In this section, we present various numerical examples to show the effectiveness of our CB-ROM. First, the data-driven POD method is employed to construct the covariance matrix. We consider the scenarios including full-aperture incidence, limited-aperture incidence, extension to multiple scatterers, and simulations with noisy data. Then, we report numerical results obtained using the physics-informed covariance formula. Finally, several experiments are carried out to evaluate the performance of the proposed GNN.

To avoid repetition, we remark here that the comparison of the CB-ROM and the conventional BIE is performed in the following numerical experiments. Specifically, for a given scatterer and incident wave, we compute the relative error of the near field at a prescribed point using both the CB-ROM (with varying numbers of reduced bases derived from the obtained covariance matrix) and the BIE (with varying numbers of boundary nodes). The comparison is illustrated through line plots showing the relative error (vertical axis) as a function of the number of reduced bases (for CB-ROM) or boundary nodes (for BIE). The incident wave direction (indicated by a blue arrow) and the observation point (marked by a black star) are clearly specified in each case.
\subsection{Random POD approach}
In all experiments involved in this subsection, the covariance matrix is constructed using the random POD method introduced in Section \ref{POD}. The wavenumber is $k=2\pi$ unless otherwise specified. 
\subsubsection{Case of single scatterer}
The first experiment focuses on a single pentagram-shaped scatterer. We first construct the correlation matrix in the full-aperture case, where the heatmap of its magnitude is exhibited in Figure \ref{f9}(a). We observe that the constructed correlation matrix obviously captures the five prominent features of the given scatterer. To illustrate the reduction effectiveness of the proposed CB-ROM, we compare it with the BIE in Figure \ref{f9}(b). The numerical results demonstrate that our CB-ROM achieves a significantly faster convergence rate. In other words, to attain the same level of precision, the CB-ROM operates on a smaller-scale linear system to obtain the contrast source. We also perform the comparison in the limited-aperture case where the incidence angle is restricted in $(\pi,\frac{3\pi}{2})$. The corresponding numerical results are depicted in Figure \ref{f9}(c)--(d). In this setting, the contrast sources at different points show stronger correlation, and our CB-ROM is proved to be more powerful.
\begin{figure}[htbp]
    \centering
    % \subfloat[]{\includegraphics[width=0.25\linewidth]{pic/starmapfull1.jpg}}
    \subfloat[]{\includegraphics[width=0.25\linewidth]{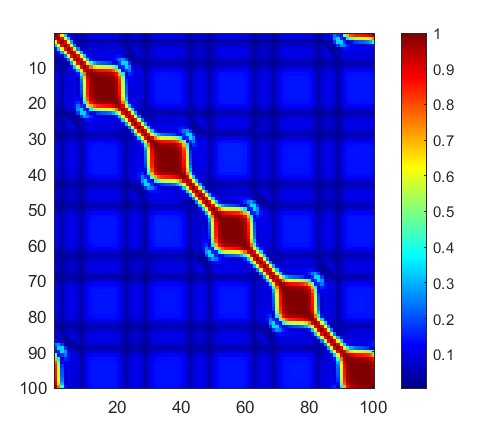}}
    \subfloat[]{\includegraphics[width=0.25\linewidth]{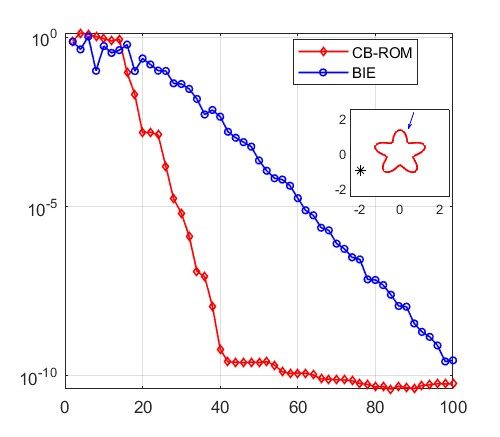}}
    \subfloat[]{\includegraphics[width=0.25\linewidth]{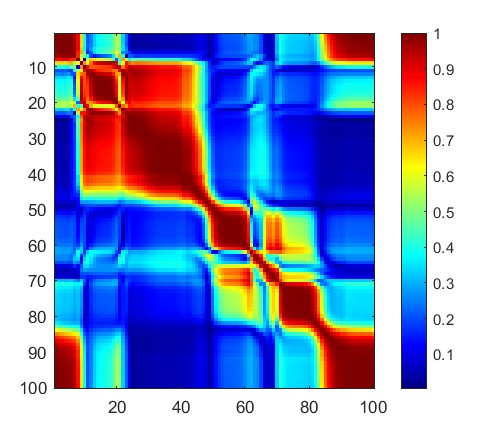}}
    % \subfloat[]{\includegraphics[width=0.25\linewidth]{pic/starmaplimited1.jpg}}
    \subfloat[]{\includegraphics[width=0.25\linewidth]{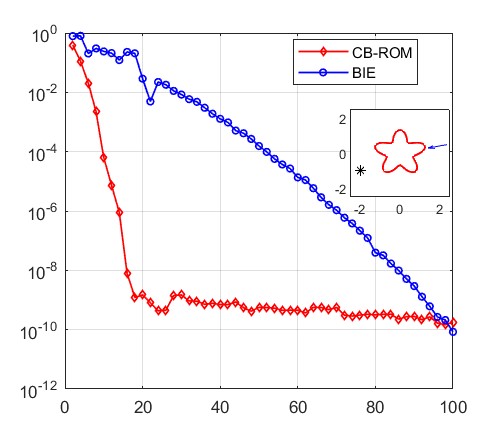}}
    \caption{(a) The magnitude of the correlation matrix for a pentagram-shaped scatterer in the full-aperture incidence case. (b) The comparison of the CB-ROM based on the covariance matrix obtained in the full-aperture incidence and the BIE. (c) and (d) are results corresponding to the limited-aperture incidence in $(\pi,\frac{3\pi}{2})$.}
    \label{f9}
\end{figure}
\begin{figure}[htbp]
    \centering
    \subfloat[]{\includegraphics[width=0.3\linewidth]{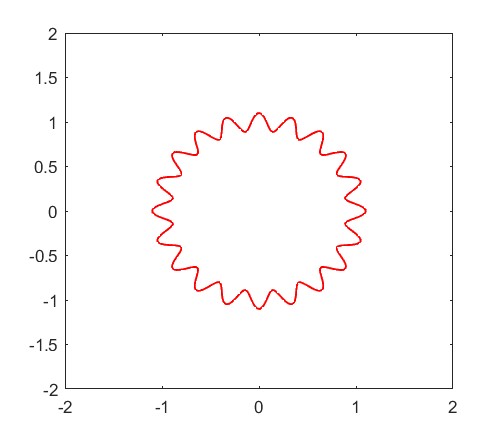}}\quad
    \subfloat[]{\includegraphics[width=0.3\linewidth]{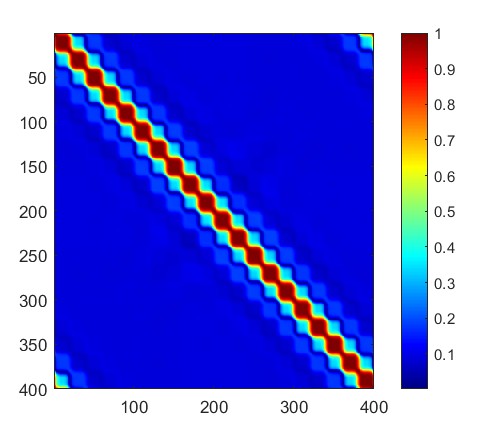}}\quad
    \subfloat[]{\includegraphics[width=0.3\linewidth]{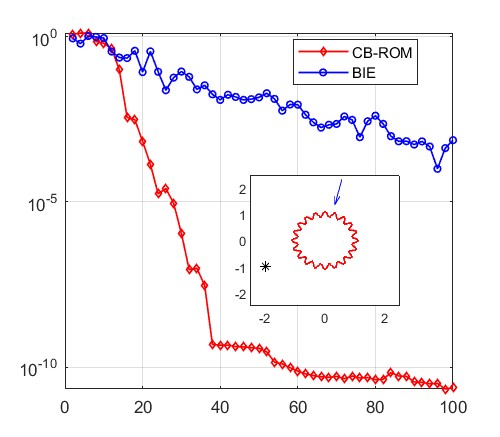}}
    \caption{(a) Shape of the scatterer with 20 convex bumps. (b) The magnitude of the correlation matrix. (c) The comparison of the CB-ROM based on the covariance matrix obtained in the full-aperture incidence and the BIE.}
    \label{f10}
\end{figure}
\begin{figure}[htbp]
    \centering
    \subfloat[]{\includegraphics[width=0.3\linewidth]{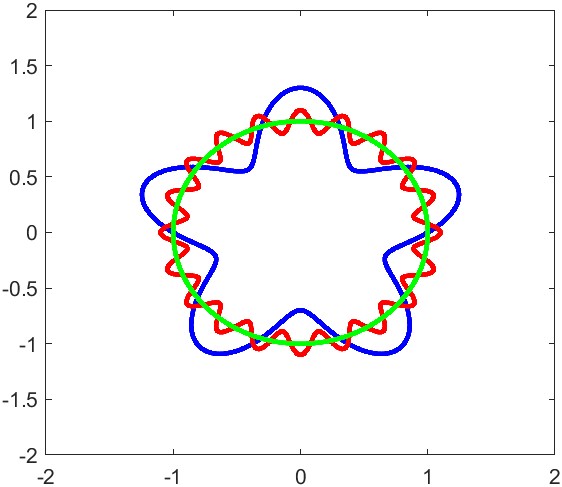}}\quad
    \subfloat[]{\includegraphics[width=0.3\linewidth]{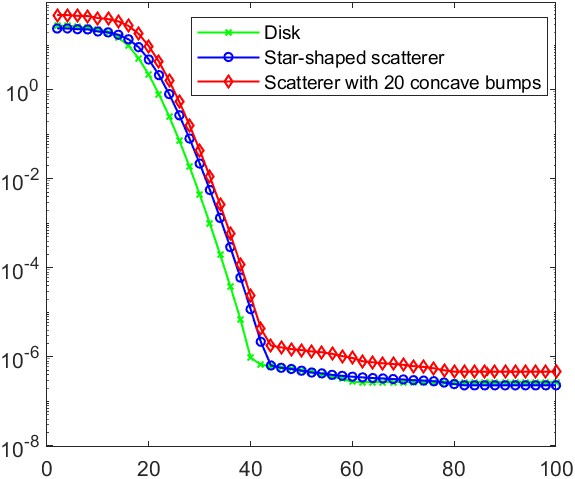}}\quad
    \subfloat[]{\includegraphics[width=0.3\linewidth]{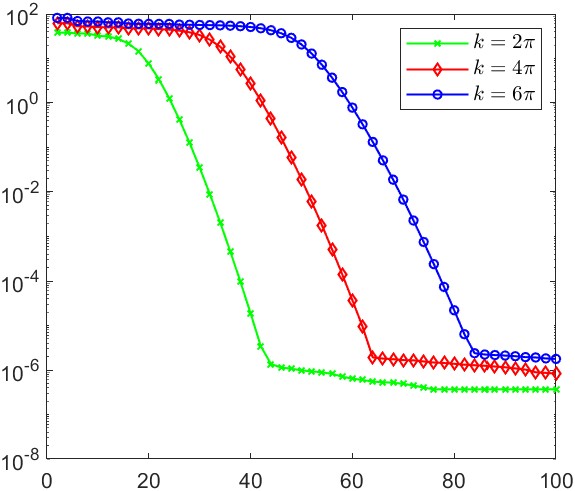}} 
    % \subfloat[]{\includegraphics[width=0.23\linewidth]{pic/exp1highfre.jpg}}
    \caption{(a) Three scatterers with comparable characteristic size but different boundary complexity. (b) The convergence factor $\sqrt{\varepsilon_{\mn}}=\sqrt{\sum_{n=\mn}^N\lambda_n}$ for three scatterers in (a) under the full-aperture incidence, where the horizontal coordinate indicates $K$.  (c) The convergence factor $\sqrt{\varepsilon_{\mn}}$ for the pentagram-shaped scatterer under three different wavenumbers, i.e., $k=2\pi,\,4\pi,\,6\pi$.}
    \label{f11}
\end{figure}
We then consider a more complex scatterer, a disk with 20 convex bumps shown in Figure \ref{f10}(a), in the full-aperture case. The yielding numerical results are reported in Figure \ref{f10}(b)--(c). These results show that the correlation matrix obtained by the POD identifies the geometric features of the scatterer. Moreover, comparing Figure \ref{f9}(b) with Figure \ref{f10}(c), we observe that the convergence rate of the CB-ROM for these two scatterers is nearly unchanged, whereas the performance of the BIE is degraded. To further illustrate this stability, we perform the CB-ROM for three scatterers with the same characteristic size but different boundary complexity under the full-aperture incidence scenario, and evaluate the convergence factor $\sqrt{\varepsilon_{\mn}}$. The numerical results, as shown in Figure \ref{f11} (b), indicate that the convergence rates for these three scatterers remain relatively stable, implying that the dominant eigenvalues of their covariance matrices are comparable. For comparison, we examine the convergence factor for the pentagram-shaped scatterer under incident waves at three different frequencies, with the corresponding results presented in Figure \ref{f11}(c). We observe that the number of dominant eigenvalues increases with the frequency. This behavior is consistent with the classical Rayleigh limitation, according to which the number of dominant eigenvalues is approximately proportional to the scatterer size divided by half the wavelength.

\subsubsection{Case of multiple scatterers}
In this subsection, we present a numerical experiment involving multiple scatterers. Specifically, $D$ consists of nine disjoint scatterers, as illustrated in Figure \ref{f12}(a), with each scatterer discretized by $50$ point. Then, we employ the POD to obtain its correlation matrix, whose magnitude is visualized in Figure \ref{f12}(b). This result suggests that the correlation matrix effectively encodes the multiple scattering effects. Similar to the case of a single scatterer, we compare the CB-ROM and the BIE for near-field computation, with numerical results presented in Figure \ref{f12}(c). We observe that the BIE requires almost all 450 degrees of freedom to achieve a precision of $10^{-7}$, whereas the CB-ROM can attain the same accuracy with less than $50$ reduced bases. This efficiency stems from the fact that bases from the correlation matrix capture the principal modes of the multiple scattering phenomenon under plane wave illumination.
\begin{figure}[htbp]
    \centering
    \subfloat[]{\includegraphics[width=0.3\linewidth]{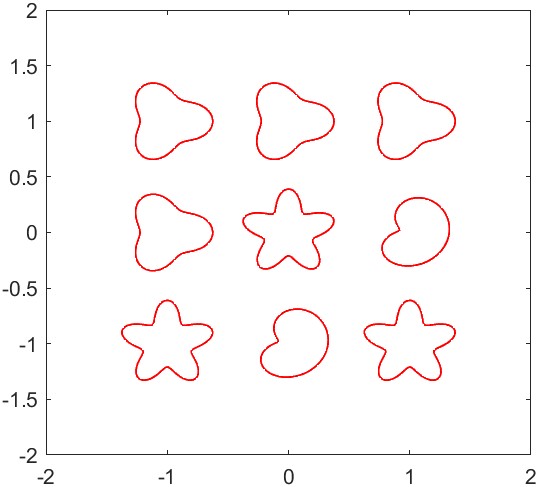}}\quad
{\includegraphics[width=0.3\linewidth]{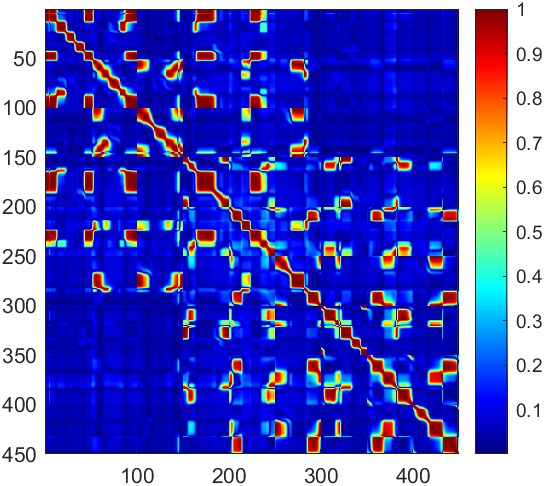}}\quad
    \subfloat[]{\includegraphics[width=0.308\linewidth]{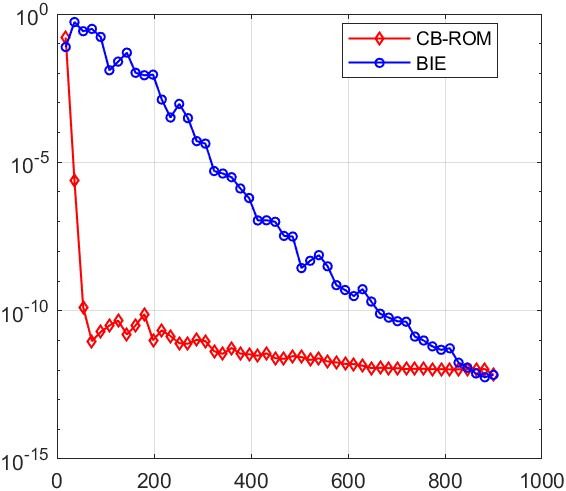}}
    \caption{(a) The case of nine scatterers. (b) The magnitude of the correlation matrix. (c) The comparison of the CB-ROM based on the covariance matrix obtained in the full-aperture incidence and the BIE.}
    \label{f12}
\end{figure}

\begin{remark}\label{cost}
    \textbf{Comparison of computation cost between BIE and CB-ROM:} Assume that there exists $Q$ incident waves with corresponding scattered field needed to be computed at $P$ observation points. Let $N$ denote the number of discretization points in BIE required to achieve a prescribed accuracy $\epsilon$. Since the inverse of $\mathbb{G}$ in \eqref{eq:highf} can be precomputed during the offline stage at one-time cost $\mathcal{O}(N^3)$, then the online computation cost for each incidence is $\mathcal{O}(N^2+PN)$.

    For the reduced-order model \eqref{rom1}, assume that $K$ bases are chosen in the proposed CB-ROM. During the offline stage, we first precompute the scattered fields at given $P$ observation points for each density basis 
    % , forming a matrix $A\in\mathbb{C}^{P\times K}$ 
    incurring a cost of $\mathcal{O}(PNK)$, followed by the inverse of $\mathbb{G}^{\rm ROM}$ at a cost of $\mathcal{O}(K^3)$.  In the online procedure, for each incident wave, it is first projected onto the basis space at a cost of $\mathcal{O}(NK)$, then the coefficients are determined at a cost of $\mathcal{O}(K^2)$, and finally field values are computed at a cost of $\mathcal{O}(PK)$. This results in a total online cost of $\mathcal{O}(NK+K^2+PK)$. While $Q$ is sufficiently large, the online stage dominates the total computational effort. In this regime, the complexity ratio between the BIE and proposed CB-ROM scales as $\mathcal{O}\left((\frac{N}{K})^2\right)$, highlighting the significant speedup offered by the CB-ROM. 
 % which means ROM saves considerable computational cost if the high-quality basis space is obtained satisfying $K<<N$.  
 
% As mentioned in Remark \ref{cost}, the computation complexity ratio of the BIE and the CB-ROM is $\mathcal{O}\left((\frac{N}{{K}})^2\right)$, where ${K}$ indicates the number of the reduced bases in our CB-ROM. 
Recalling the above example involved the scatterer with $20$ convex bumps, we can see from Figure \ref{f10}(c) that the precision of $10^{-3}$ corresponds to $N=100$ and ${K}=20$. It indicates the computation cost of the BIE is $25$ times of that of the CB-ROM. The numerical results shown in this section suggest that this acceleration advantage is further enhanced in the multiple-scatterers case and the limited-aperture case.
\end{remark}

\subsubsection{Nosiy case}\label{noise}
Consider adding a noise random matrix $\Delta C$ with noise level $\|\Delta C\|_2\approx10^{-2}$ to the high-fidelity covariance matrix for different shapes. As seen in Figure \ref{f13}, the near-field error keeps decaying until the eigenvalue is near the noise level, which is consistent with our stability analysis in Section \ref{stable} and show the robustness of the CB-ROM. 
\begin{figure}[H]
    \centering
    \subfloat[]{\includegraphics[width=0.3\linewidth]{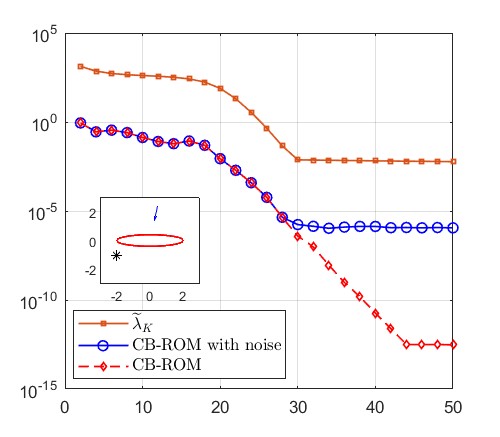}}\quad
    \subfloat[]
{\includegraphics[width=0.3\linewidth]{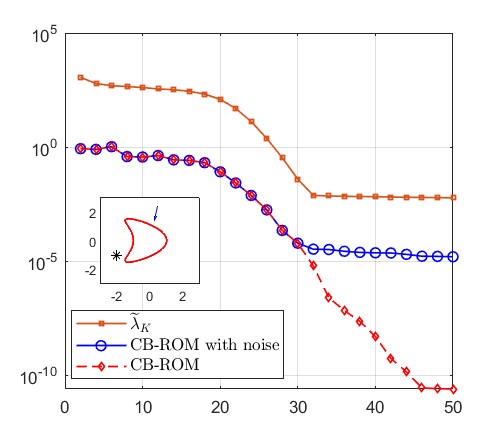}}\quad
    \subfloat[]{\includegraphics[width=0.3\linewidth]{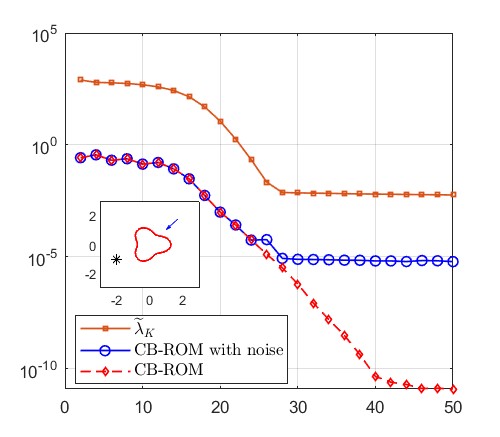}}
    \caption{Noise level $\sigma=10^{-2}$ with $\tilde{\lambda}_\mn$ denoting the $\mn$-th eigenvalue of the noisy covariance matrix $\tilde{C}$. }
    \label{f13}
\end{figure}
\subsection{Physics-informed intuitive formula}\label{subsec:PI-formula-num}
We conduct some numerical experiments to demonstrate the effectiveness of the physics-informed covariance formula. That is, we first apply the formulas introduced in Section \ref{subsec:pi-formula} to generate the covariance matrix for cases including a single scatterer and multiple scatterers, respectively. The near field is then calculated using the CB-ROM and the BIE, with numerical results visualized in Figure \ref{f14}(a)--(b). The convergence rate of the CB-ROM remains faster than that of the BIE. Furthermore, we consider a scatterer with a complex boundary as shown in Figure \ref{f14}(c). In this case, the capacity of the CB-ROM is slightly weakened, but still suggests that the physics-informed formula effectively captures the low-frequency representation of the contrast source, thereby providing the first few dominant modes.
\begin{figure}[htbp]
    \centering
    \subfloat[]{\includegraphics[width=0.3\linewidth]{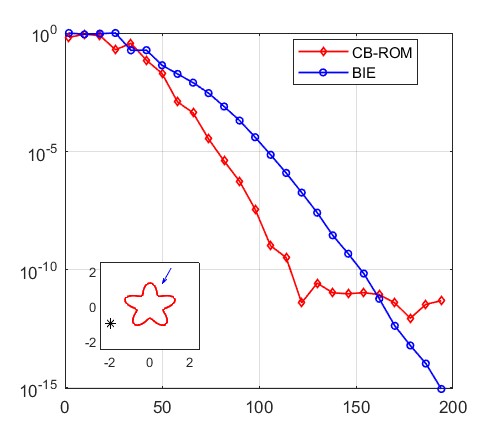}}\quad
    \subfloat[]{\includegraphics[width=0.3\linewidth]{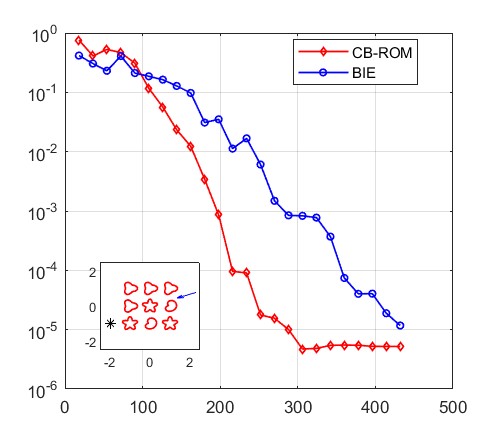}}\quad
    \subfloat[]{\includegraphics[width=0.3\linewidth]{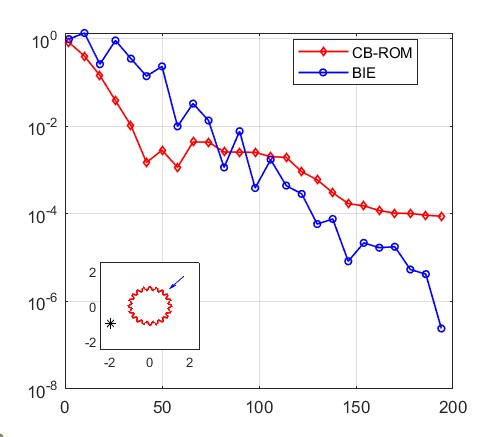}}
    \caption{The comparison of the CB-ROM based on the covariance matrix obtained by the physics-informed formula and the BIE. In (a) and (b), the wavenumber $k=4\pi$. (c) The wavenumber $k=2\pi$.}
    \label{f14}
\end{figure}

\subsection{Numerical results of the GNN}
In this subsection, we aim to construct the reduced bases for our CB-ROM via the covariance matrix learned by the proposed GNN introduced in Section \ref{sec:gnn}. The numerical setups are specified as follows. First, we focus on the full aperture case in this subsection, i.e., the scatterer is illuminated by $1000$ plan waves with wavenumber $k = 8\pi$ and incident directions admitting a uniform distribution in $[0,\,2\pi)$. Second, we generate $3200$ samples, in which $3000$ constitute the training set ${T}_t$ and the remaining $200$ comporise the test set ${T}_e$. Each sample $D$ consists of one star-like obstacle, whose parameteric form is given by
\begin{equation}\label{eq:trainset}
    \partial D = \{r(t) (cos(t), sin(t)),\,t\in[0,2\pi]\},\quad  r(t) = r_0+\sum^{N_k}_{k=1}\left(a_k sin(t)+b_k cos(t)\right)
\end{equation}
with $N_k = 5$. We pick the basic radius $r_0$ uniformly at random in the interval $[0.8,\,1.2]$. Moreover, the Fourier coefficients $a_k,\,b_k\,(k=1,\ldots,5)$ are sampled independently in $[-0.1,\,0.1]$. Third, the true covariance matrix, denoted as $\Sigma(D)\in\mathbb{C}^{100\times 100}$, for each scatterer $D$ is obtained via POD with $M=1000$ and $N=100$ in \eqref{eq:cov-num}. In particular, we use the combined potential \eqref{equ:integratialrep-com} for solving the contrast source, with $\eta$ taken as the wavenumber. Finally, we specify the adopted loss function in the training process. The classical mean square error (MSE) is first used to punish the local point-wise discrepancy of the covariance matrix. To further capture the structural feature of the target matrix, the Structural Similarity Index Measure (SSIM) \cite{SSIM-2004}, which is widely applied to extract the global features of images, is added into our loss function. Consequently, the empirical objective based on the training set is to find the optimal $\Theta^*$ by solving the optimization problem
\begin{equation}\label{eq:lossfun}
    \min_{\Theta}\mathcal L(\Theta):= \frac{1}{|{T}_t|}\sum_{D\in {T}_t} \left(\frac{1}{2}\rm{MSE}\left(\mathcal F_\Theta(D), \Sigma(D)\right) + \frac12\rm{SSIM}\left(\mathcal F_\Theta(D), \Sigma(D)\right)\right).
\end{equation}
\begin{figure}[htbp]
    \centering
    \subfloat{\includegraphics[width=0.25\linewidth]{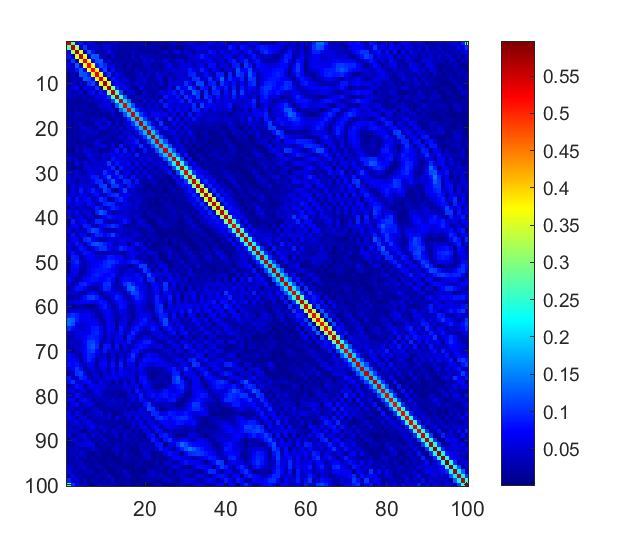}}
    \subfloat{\includegraphics[width=0.25\linewidth]{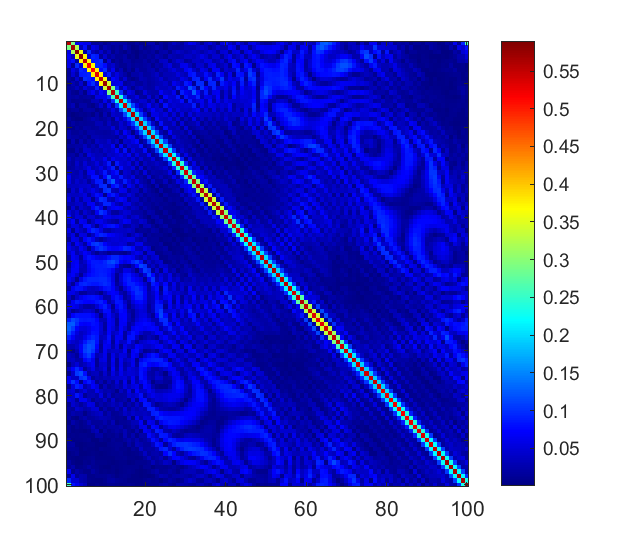}}
    \subfloat{\includegraphics[width=0.25\linewidth]{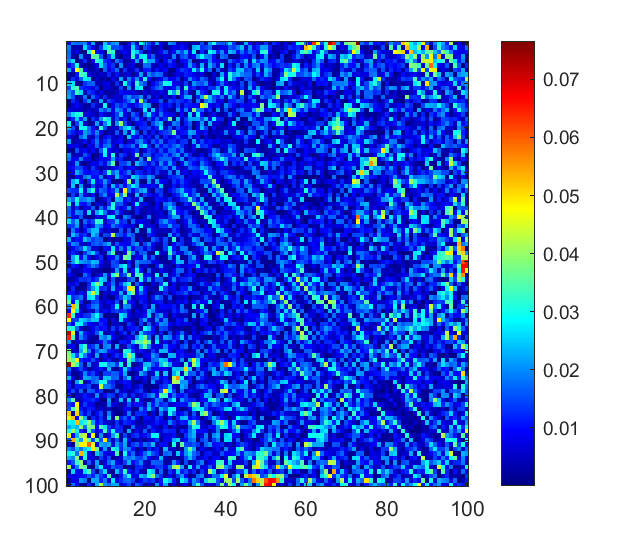}}
    \subfloat{\includegraphics[width=0.28\linewidth]{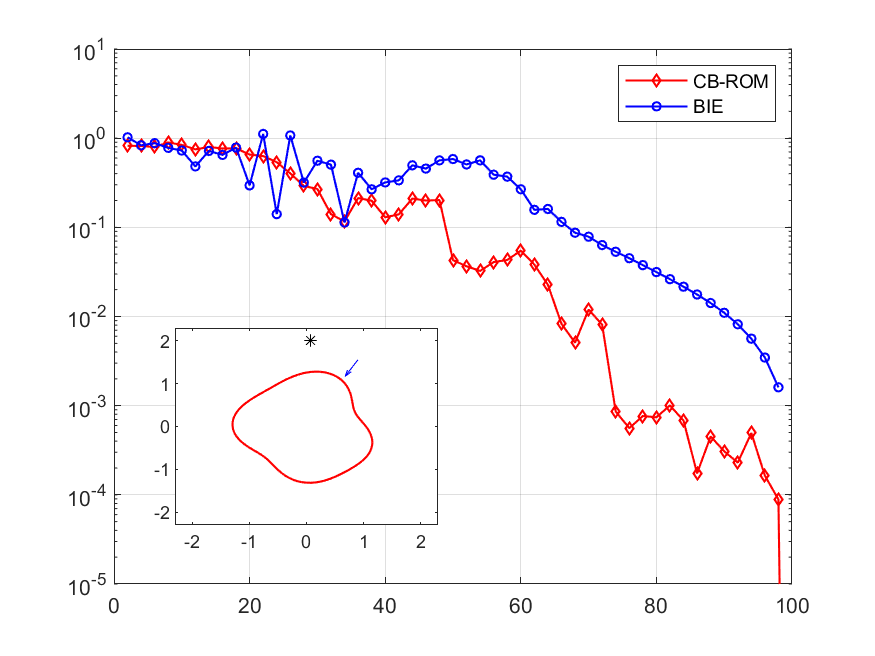}}

    \subfloat{\includegraphics[width=0.25\linewidth]{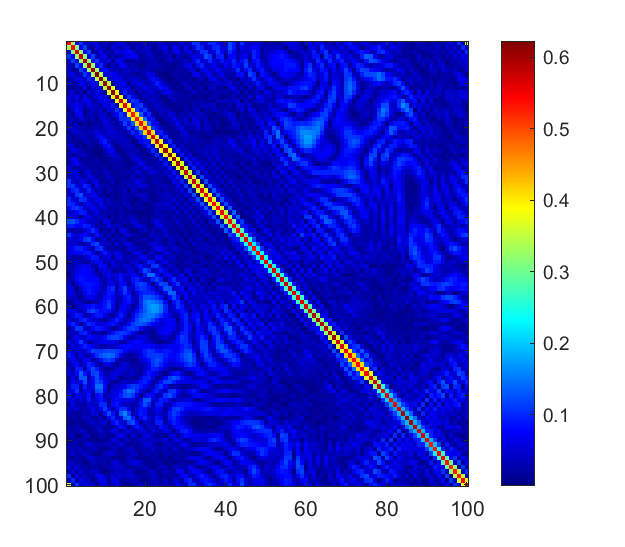}}
    \subfloat{\includegraphics[width=0.25\linewidth]{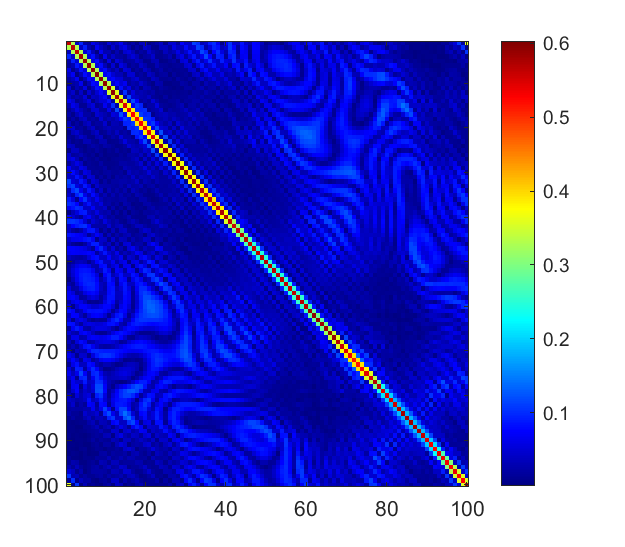}}
    \subfloat{\includegraphics[width=0.25\linewidth]{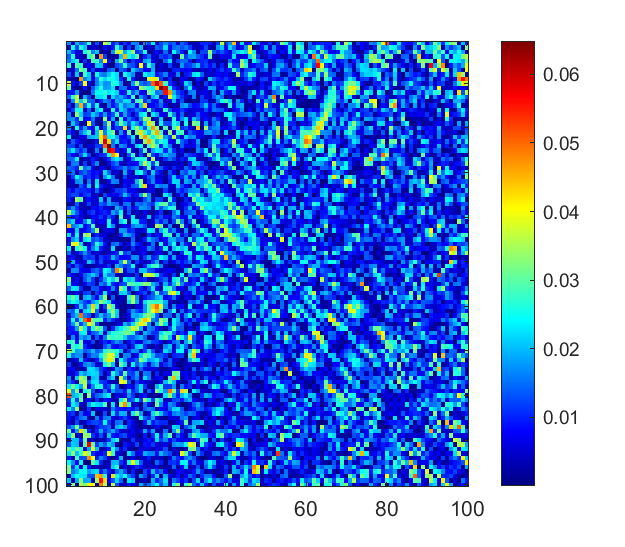}}
    \subfloat{\includegraphics[width=0.28\linewidth]{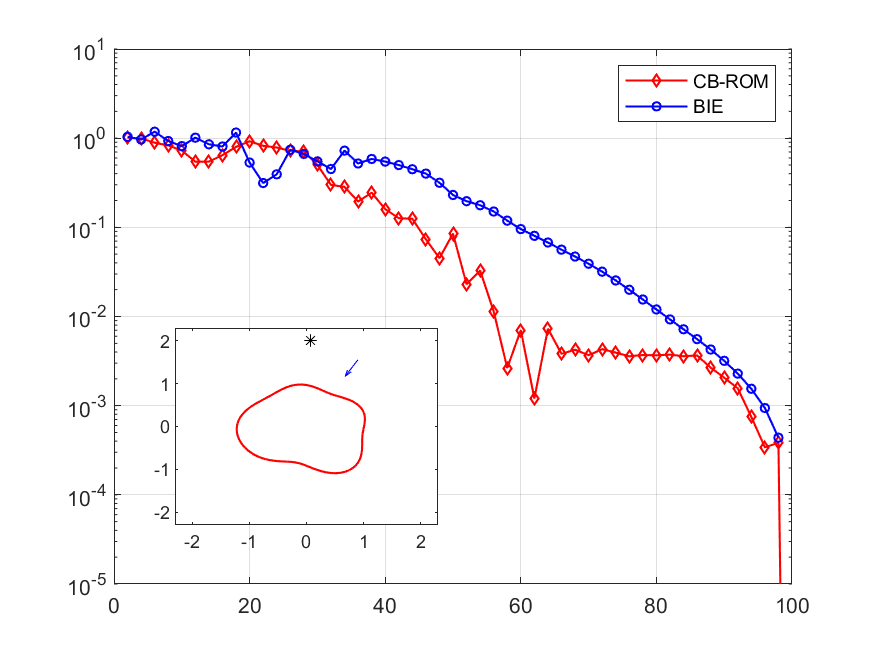}}

    \subfloat{\includegraphics[width=0.25\linewidth]{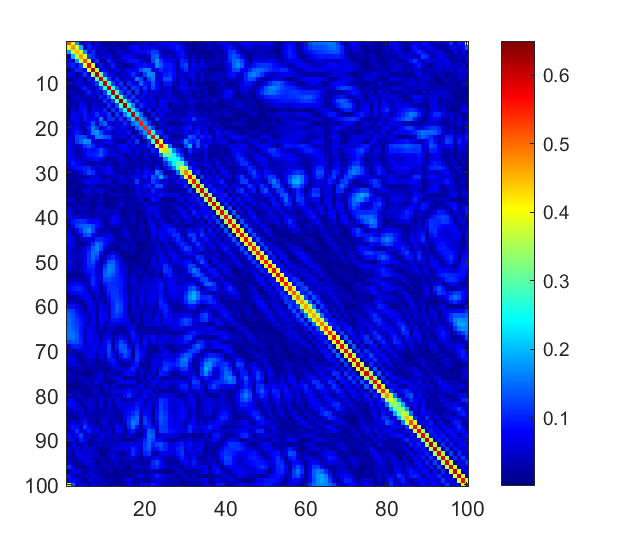}}
    \subfloat{\includegraphics[width=0.25\linewidth]{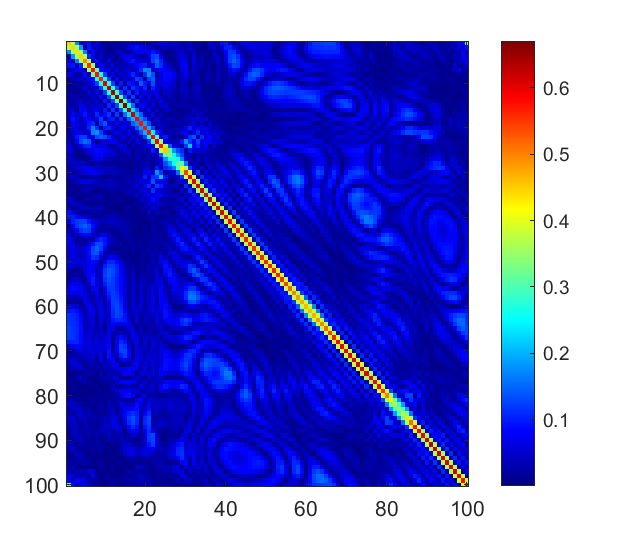}}
    \subfloat{\includegraphics[width=0.25\linewidth]{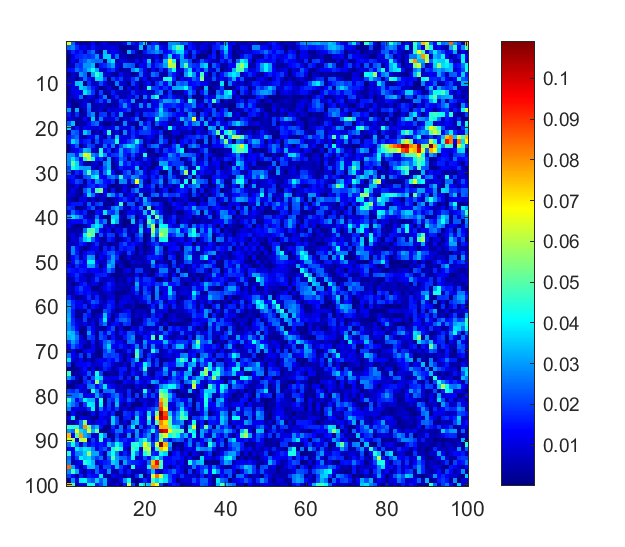}}
    \subfloat{\includegraphics[width=0.28\linewidth]{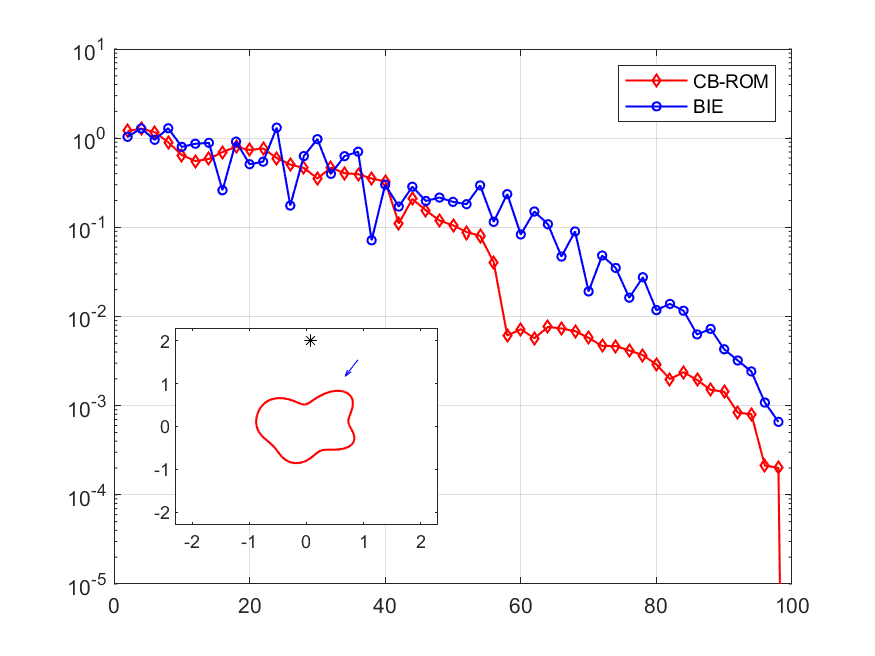}}
    \caption{Numerical results for three representative samples in the test set (first to third rows). For each sample, we present the magnitude of the true covariance matrix (first column), the covariance matrix learned by the proposed GNN (second column), and the absolute error between them (third column). The last column compares the CB-ROM based on the learned covariance matrix and the BIE.}
    \label{GNN-Results}
\end{figure}

Next, let us describe the graph-level details. Each node feature vector is denoted as $\bm x^v_i = [h^i_x,\,h^i_y,\,\kappa_i]^T$, where $(h^i_x,\,h^i_y)$ and $\kappa_i$ indicate the coordinate and curvature of the corresponding node, respectively. As mentioned in Section \ref{sec:gnn}, we need to initialize the edge features for all edges. It is noteworthy that the covariance matrix for a unit disk is closely related to $e^{i(\theta-\theta^\prime)}$, where $(\theta-\theta^\prime)$ is the difference of the discretized angles corresponding to two nodes; see \eqref{eq:covdisk} for details. Inspired by this observation, we initialize each edge $\bm x^e_{ij}$ as $[\cos\beta_{ij}, \sin \beta_{ij}]^T$ with $\beta_{ij}$ the angle difference between $v_i$ and $v_j$. 

We highlight that the covariance matrix $\mathcal F_{\Theta^*} (D)$ generated by the trained GNN might not exactly match the true covariance matrix $\Sigma(D)$. The possible small discrepancy introduces high-frequency noise in the exact covariance matrix, which may significantly affect the small singular values and the corresponding singular vectors. To mitigate the issue, we append a lightweight CNN \cite{GNN-Denoisor-2022} after the GNN to perform an additional denoising step. For notational simplicity, we continue to denote the combined output for a given $D$ as $\mathcal F_{\Theta^*} (D)$. In addition, the mean density function is required to construct the complete reduced basis. We assume that the mean function is known a priori. In practice, it is easily obtained by solving a linear system \eqref{eq:boundary} once.

\begin{figure}[htbp]
    \centering
    \subfloat{\includegraphics[width=0.25\linewidth]{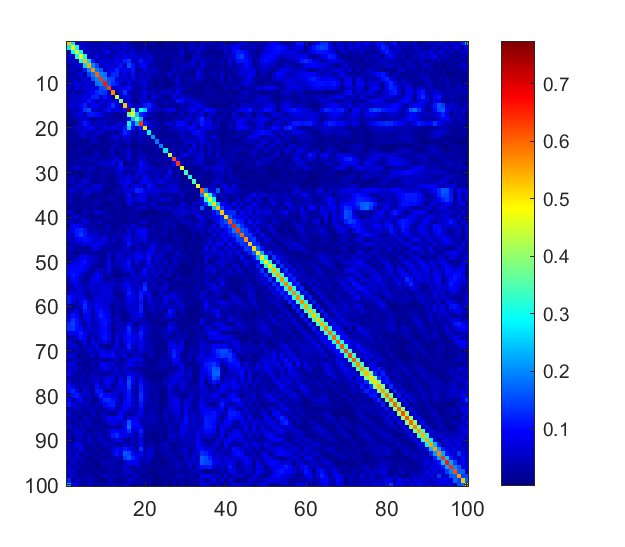}}
    \subfloat{\includegraphics[width=0.25\linewidth]{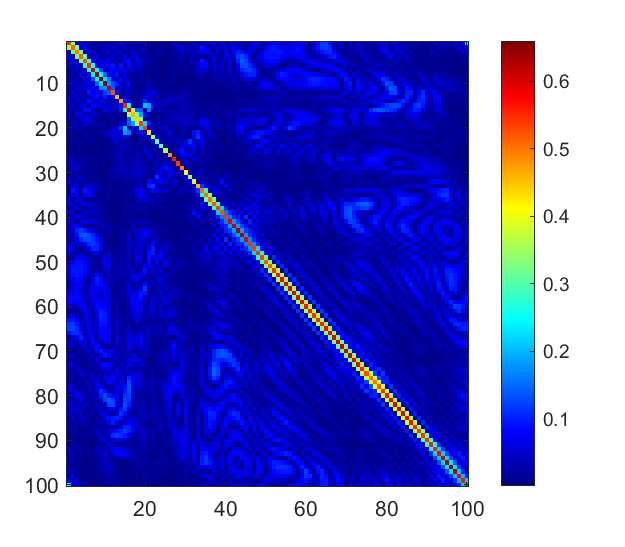}}
    \subfloat{\includegraphics[width=0.25\linewidth]{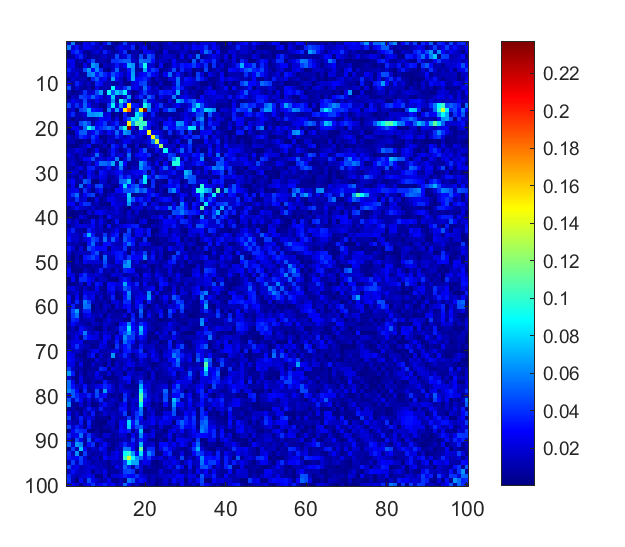}}
    \subfloat{\includegraphics[width=0.28\linewidth]{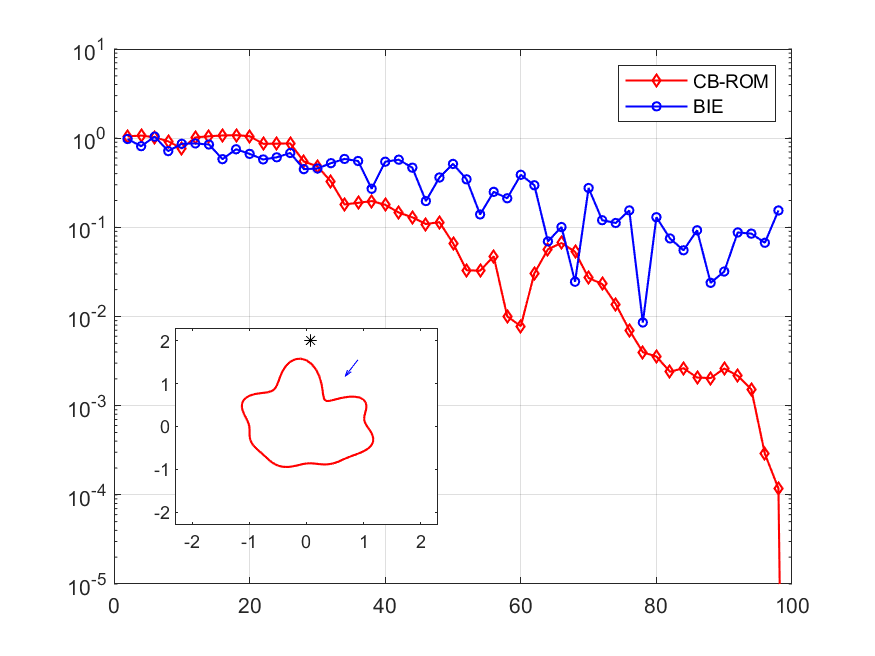}}

    \subfloat{\includegraphics[width=0.25\linewidth]{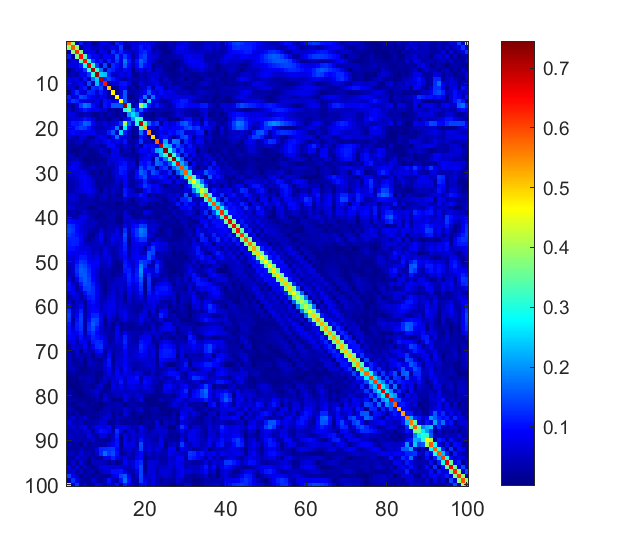}}
    \subfloat{\includegraphics[width=0.25\linewidth]{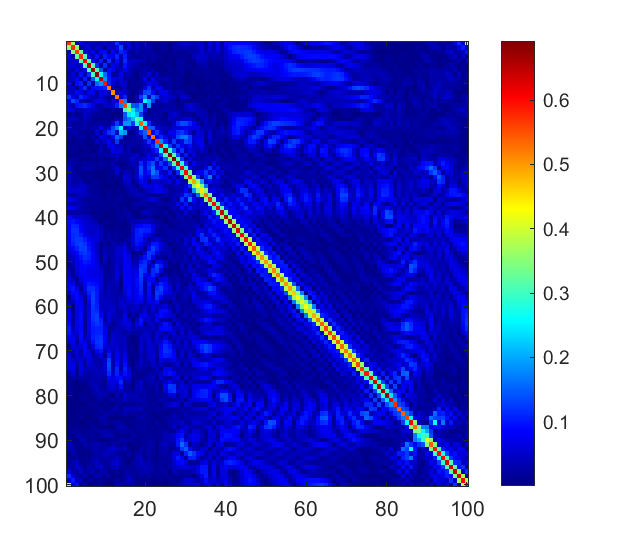}}
    \subfloat{\includegraphics[width=0.25\linewidth]{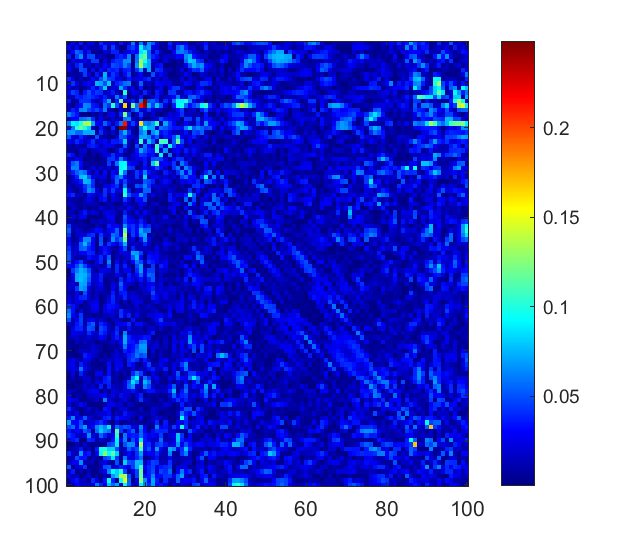}}
    \subfloat{\includegraphics[width=0.28\linewidth]{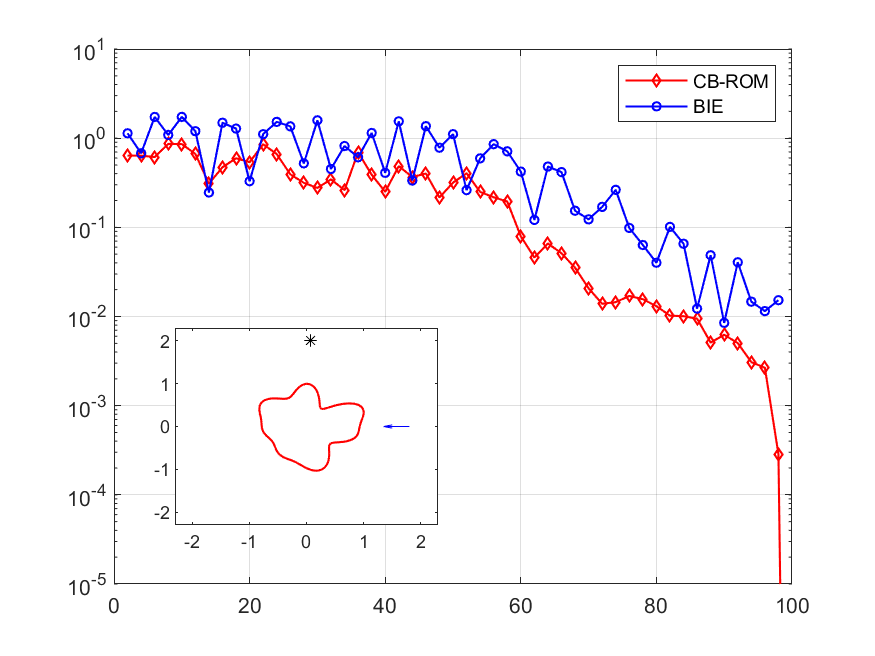}}

    \subfloat{\includegraphics[width=0.25\linewidth]{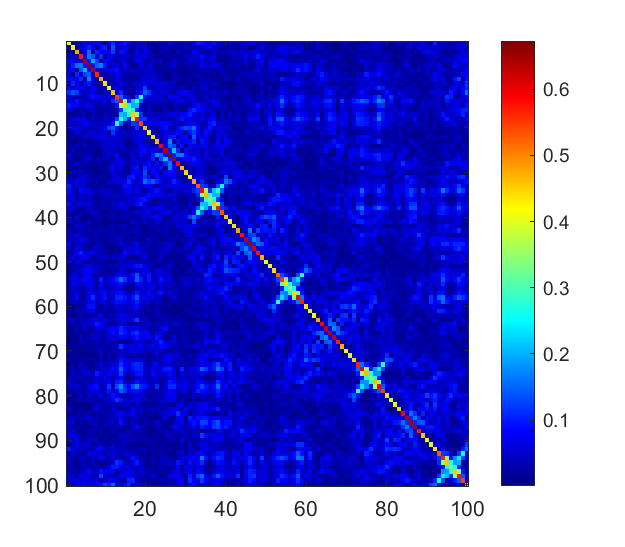}}
    \subfloat{\includegraphics[width=0.25\linewidth]{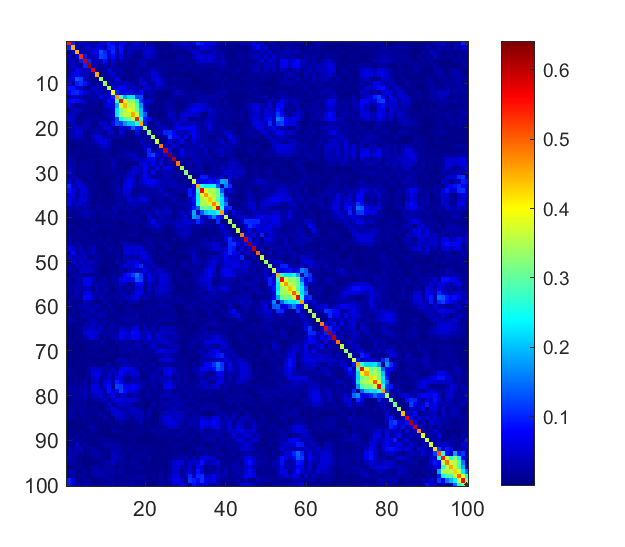}}
    \subfloat{\includegraphics[width=0.25\linewidth]{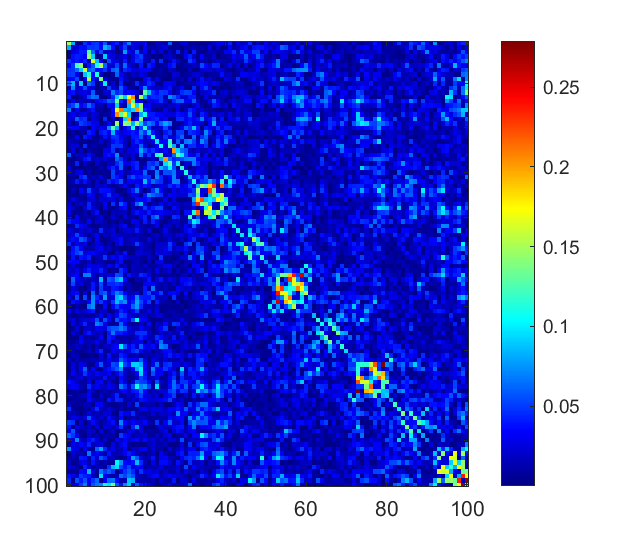}}
    \subfloat{\includegraphics[width=0.28\linewidth]{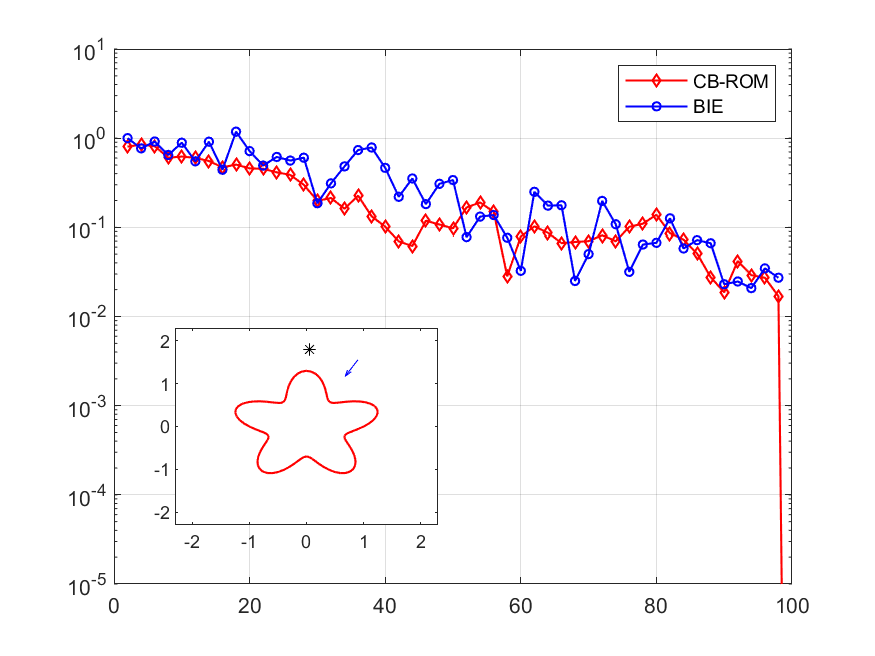}}
    \caption{Numerical results for testing the generalization capacity of the GNN. The first two rows shows the resutls for two samples generated by \eqref{eq:trainset} with $N_k = 7$, and the last row corresponds to a pentagram-shaped scatterer. }
    \label{GNN-Results-Generalization}
\end{figure}

After training, we evaluate the trained model by constructing the covariance matrices for three distinct random samples in the test set. Then, we simulate the scattering of the plane wave, incident from the direction $(\cos(\frac{\pi}{3}),\,\sin(\frac{\pi}{3}))$, through the selected samples. The scattered field is then calculated using the CB-ROM based on the learned covariance matrices. The numerical results are depicted in Figure \ref{GNN-Results}. For each row, the first column shows the ground truth covariance matrix, the second column displays the learned one with our GNN, and the third column quantifies the point-wise absolute error between them. We observe that the absolute error is of order $\mathcal O(10^{-2})$, which shows the accuracy of our trained GNN. To assess the effectiveness of the learned covariance matrix, we compare the CB-ROM based on the learned covariance matrix against the BIE. The comparative results are presented in the last column of Figure \ref{GNN-Results}. It is evident that the CB-ROM using the learned covariance matrix significantly outperforms the BIE. These findings collectively demonstrate the efficacy of our GNN in extracting the underlying covariance matrix for a given scatterer. 

Furthermore, to illustrate the generalization capacity of the proposed GNN, we examine its performance on OOD samples. We first generate two samples using \eqref{eq:trainset} with $N_k = 7$. These two samples exhibit more complex boundaries compared with those in the training set, where $N_k=5$. The third sample is a pentagram-shaped scatterer, which is clearly beyond the training distribution. We use the trained GNN to construct the covariance matrices for these three samples and perform the CB-ROM, with numerical results displayed in Figure \ref{GNN-Results-Generalization}. The results show that the learned GNN constructs the covariance matrices reasonably well for these OOD samples, thereby demonstrating the good generalization of our method. Moreover, it is worth noting that the BIE fails to solve the near field precisely with only $100$ nodes for these complex scatterers, while the CB-ROM based on the learned covariance matrix achieves a high accuracy with a stable convergence rate. 

\section{Conclusion}\label{sec:conclusion}
In this paper, we proposed a covariance-based reduced-order model (CB-ROM) for solving acoustic scattering problems. By modeling the contrast source under random incidence as a spatial random field, we established a rigorous statistical framework for the ROM via the Karhunen-Loève (KL) expansion. The CB-ROM adopts the dominant eigenfunctions of the covariance operator as reduced bases, efficiently capturing the principal modes of the contrast source for a given incidence distribution. Furthermore, we established the convergence properties and numerical stability of the proposed method in expectation. To construct the covariance matrix for an arbitrary scatterer, we developed three complementary strategies: a data-driven proper orthogonal decomposition (POD) approach (highly accurate but computationally intensive and case-specific), a physics-informed analytical formulation (computationally simple but of limited fidelity), and a learning-based graph neural network (GNN) approach (highly efficient and adaptive). Extensive numerical experiments demonstrated the efficacy and robustness of the proposed CB-ROM framework.

This work motivates several promising directions for future research. First, the CB-ROM framework can be extended to scattering problems in heterogeneous media. Second, the methodology is naturally applicable to a broader class of wave phenomena, including elastic and electromagnetic scattering, and holds particular promise for high-dimensional scattering configurations where rapid computational methods are strictly required. Finally, by exploiting the correlation structure of the contrast source, the CB-ROM provides valuable physical insights into multiple scattering mechanisms. This statistical perspective may not only deepen our understanding of resolution limits but also presents new opportunities for advancing super-resolution imaging techniques in inverse scattering problems.

\bibliography{bib.bib}
\bibliographystyle{abbrv} % 使用缩写作者名的样式
\end{document}